\documentclass[final]{siamltex}

\usepackage{graphicx}
\usepackage{color}
\usepackage{xcolor}
\usepackage{bm}
\usepackage{amssymb}
\usepackage{amsfonts}
\usepackage{amsmath}
\usepackage{mathtools}
\usepackage{enumerate}
\usepackage{url}

\renewcommand{\tilde}{\widetilde}
\renewcommand{\hat}{\widehat}
\renewcommand{\bar}{\overline}

\input epsf

\newtheorem{remark}{Remark}[section]

\newtheorem{example}[theorem]{Example}
\newtheorem{alg}[theorem]{Algorithm}

\newcommand{\zhc}[1]{\textcolor{black}{#1}}
\newcommand{\rvline}{\hspace*{-\arraycolsep}\vline\hspace*{-\arraycolsep}}

\begin{document}
\title{A Generating Polynomial Based Two-Stage Optimization Method for Tensor Decomposition
\thanks{
This research was partially supported by the National Natural Science Foundation (DMS-2110722, DMS-2309549).}
}

\author{
Zequn Zheng\thanks{{\tt zzheng@lsu.edu},
Department of Mathematics,
Louisiana State University, Baton Rouge, LA USA 70803-4918.   }
\and
 Hongchao Zhang\thanks{{\tt hozhang@math.lsu.edu},
https://math.lsu.edu/$\sim$hozhang,
        Department of Mathematics,
        Louisiana State University, Baton Rouge, LA USA 70803-4918.
        Phone (225) 578-1982. Fax (225) 578-4276.}
\and 
Guangming Zhou\thanks{{\tt zhougm@xtu.edu.cn},
Hunan Key Laboratory for Computation and Simulation in Science and Engineering, School of Mathematics and Computational Science, Xiangtan University,
 Xiangtan, Hunan, China 411105.}
}
\maketitle

\begin{abstract}
The tensnor rank decomposition, also known as canonical polyadic (CP) or simply tensor decomposition, has a long-standing history in multilinear algebra.
However, computing a rank decomposition becomes particularly challenging when the tensor's rank lies between its largest and second-largest dimensions.
Moreover, for high-order tensor decompositions, a common approach is to first find a decomposition of its flattening order-3 tensor, 
where a significant gap often exists between the largest and the second-largest dimension, also making this case crucial in practice.
For such a case, traditional optimization methods, such as the nonlinear least squares or alternating least squares methods, often fail to 
produce correct tensor decompositions. 
There are also direct methods that solve tensor decompositions algebraically. 
However, these methods usually require the tensor decomposition to be unique and can be computationally expensive, especially when the tensor rank is high.
This paper introduces a new generating polynomial (GP) based two-stage algorithm for finding the order-3 nonsymmetric tensor decomposition 
even when the tensor decomposition is not unique,
assuming the rank does not exceed the largest dimension.
The proposed method reformulates the tensor decomposition problem into two sequential optimization problems. 
Notably, if the first-stage optimization yields only a partial solution, it will be effectively utilized in the second stage.
We establish the theoretical equivalence between the CP decomposition and the global minimizers of those two-stage optimization problems. 
Numerical experiments demonstrate that our approach is very efficient and robust, capable of finding tensor decompositions 
in scenarios where the current state-of-the-art methods often fail.
\end{abstract}



\pagestyle{myheadings}
\thispagestyle{plain}
\markboth{Z. ZHENG, H. ZHANG and G. ZHOU}
{Generating Polynomial Based Two-Stage Optimization for Tensor Decomposition}


\section{Introduction}
Tensors, known as higher-order generalizations of matrices, have numerous crucial practical applications. 
They have been widely used to represent multidimensional data such as parameters in a neural network and higher-order moments in statistics. 
In summary, tensors are ubiquitous in statistics \cite{TensorsStatistics,guo2021learning,RungangExactClustering,han2022optimal,Communitydetection_JLLD}, 
neuroscience \cite{Guha_Neuroimaging_2024,HuaTensorRegression}, signal processing 
\cite{Cichocki_2015,MironTDSP2020,Sidiropoulos_TD_SP2017}, and data science 
\cite{learninglatentvariable,girka2023tensor,liu2023tensor,pereira2022tensor,DynamicTensor}.

Denote $\mathcal{F} \in \mathbb{C}^{n_1\times \ldots \times n_m}$ as an order $m$ tensor with dimension $n_1,\ldots,n_m$ over the complex field. It can be represented by a multi-dimensional array 
\begin{equation*}
    \mathcal{F} = (\mathcal{F}_{i_1,\ldots,i_m})_{1\le i_1\le n_1,\ldots,1\le i_m \le n_m}.
\end{equation*}
For vectors ${u}_1 \in \mathbb{C}^{n_1},\ldots, {u}_m \in \mathbb{C}^{n_m}$, their outer product ${u}_1 \otimes {u}_2 \otimes \cdots \otimes {u}_m \in \mathbb{C}^{n_1 \times \ldots \times n_m}$ is defined as
\begin{align}
    ({u}_1 \otimes {u}_2 \otimes \cdots \otimes {u}_m)_{i_1,i_2,\cdots,i_m} = ({u}_1)_{i_1}({u}_2)_{i_2}\cdots({u}_m)_{i_m}.
\end{align}
Tensors that can be written as an outer product of $m$ nonzero vectors are called rank-1 tensors, i.e., ${u}_1 \otimes {u}_2 \otimes \cdots \otimes {u}_m$.
 For an arbitrary tensor $\mathcal{F}$, it always can be written as a summation of rank-1 tensors, that is
\begin{equation} \label{sum_pure}
    \mathcal{F} = \sum_{i=1}^r {u}^{i,1}\otimes \cdots \otimes {u}^{i,m},
\end{equation}
where ${u}^{i,j} \in \mathbb{C}^{n_j} $. The tensor rank of $\mathcal{F}$ is the smallest $r$ in \eqref{sum_pure}, denoted by $\rank(\mathcal{F})$. 
In this case, the shortest decomposition \eqref{sum_pure} is referred as the rank decomposition, or alternatively, the Candecomp-Parafac (CP) Decomposition, 
Canonical Decomposition (CANDECOMP), Parallel Factor Model (PARAFAC), or simply tensor decomposition (TD).
Unlike the rank of a matrix, the tesnors' rank can be far more complicated.
For example, it is common for the rank of a tensor to be greater than all its dimensions. 
Even determining the rank of a general tensor is an NP-hard problem \cite{HillarNPHard2013}.

Given a generic tensor $\mathcal{F}$ and its rank $r$, the tensor decomposition problem aims to find a tensor decomposition as in \eqref{sum_pure}.
This problem will get harder when the rank and order of the tensor increase. 
In this paper, we will focus on order-3 tensors with dimensions $n_1 \geq n_2 \geq n_3$. 
For high-order tensor decompositions, one may first flatten it to cubic tensors and then, find its decomposition based on cubic tensor decompositions. 
For example, a tensor $\mathcal{F}\in \mathbb{C}^{6 \times 6 \times 6 \times 6 }$ can be flattened to an order-$3$ tensor $\mathcal{M}\in \mathbb{C}^{36 \times 6 \times 6 }$ 
and then, find the tensor decomposition of $\mathcal{M}$, from which 
the decomposition of  $\mathcal{F}$ can be recovered if the tensor rank is below a certain bound.
We refer to \cite{Chiantini_2017,NWZ23,Phanreshaping2013} for the details of flattening techniques.
We could divide cubic tensors into the following three cases based on their rank $r$:
\vspace{0.05in}
\begin{center}{\footnotesize
\begin{tabular}{ c c c }
$\bullet$ Low-Rank Case: $r \leq n_2$; & $\bullet$ Middle-Rank Case: $n_2 < r \leq n_1$; & $\bullet$ High-Rank Case: $n_1 < r$. \\ 
\end{tabular}}
\end{center}
\vspace{0.05in}
The most commonly used methods for tensor decomposition are the alternating least squares (ALS) and nonlinear least squares (NLS) methods
 \cite{TensorToolbox,TensorLy,minster2021cp,TensorLab}. 
These methods often work well in practice when the tensor rank is small.
However, they usually fail to find a tensor decomposition and converge to local minimizers when the tensor rank gets higher. 
Figure~\ref{TD:NLSALS} shows the performance of ALS and NLS methods in {\tt Tensorlab} \cite{TensorLab} for finding
 decompositions of randomly generated $(20,15,10)$ tensors with different ranks, where the tensors
were generated with entries following a normal distribution with mean $1,2,3$ and variance $1$.
See the numerical experiment section for the specific settings of the ALS, NLS methods and the `success' rate.
\begin{figure}
    \centering
    \includegraphics[width=0.8\linewidth]{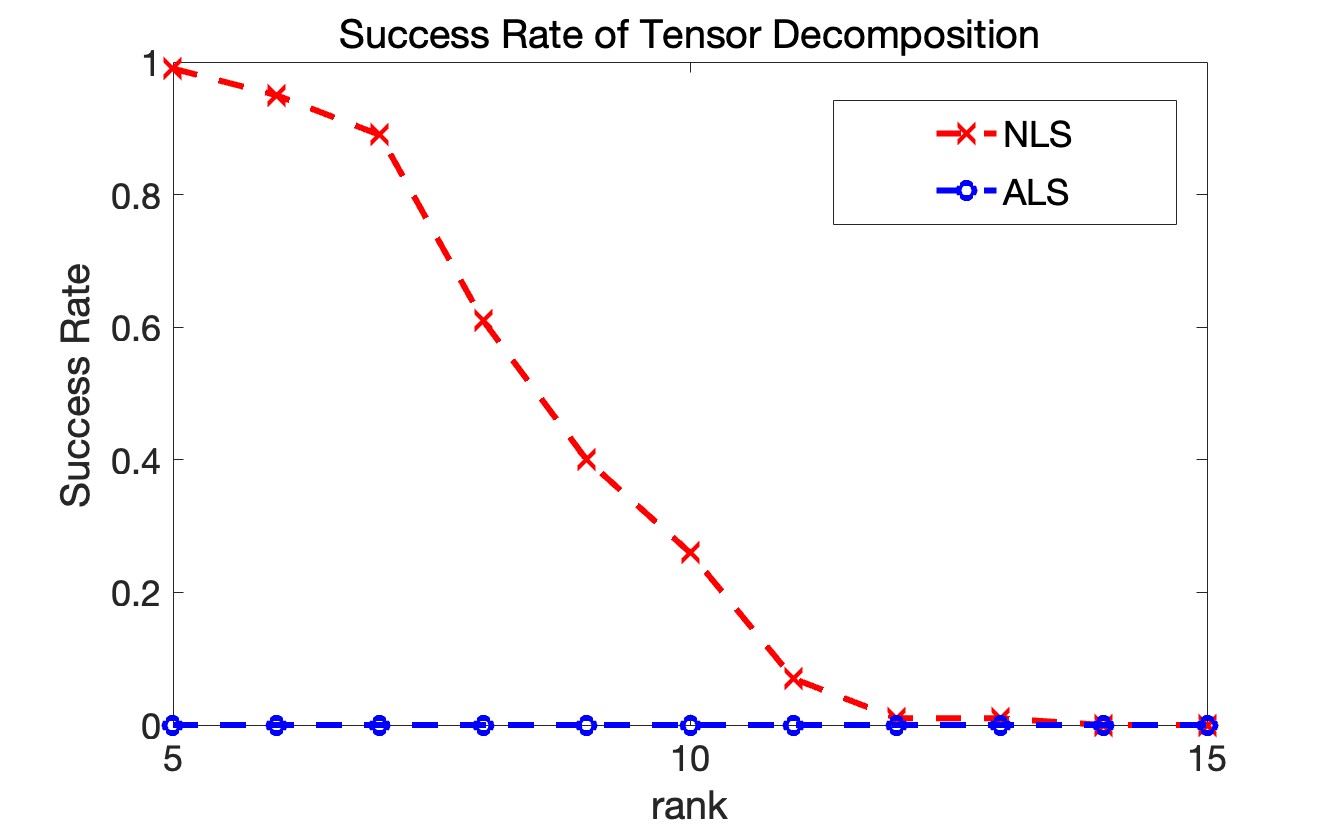}
    \caption{Tensor decomposition success rate for 100 tensors with dimension $(20,15,10)$.}
    \label{TD:NLSALS}
\end{figure}
For this example of using random starting points, the success rate of ALS for finding tensor decomposition is nearly zero, and the success rate of NLS gets close to zero when the rank increases to the Middle-Rank Case.
For a reasonable success rate, ALS usually needs a quite good starting point.
Another classical and well-known tensor decomposition method is the generalized eigenvalue decomposition (GEVD) method \cite{Leurgans1993,Sanchez1990}, 
which selects two tensor slices (also known as matrix pencils) of the tensor and computes the generalized eigenvectors of the two matrices to recover its tensor decomposition.
However, in the generic case, GEVD can only find the tensor decomposition for the Low-Rank Case.

When the tensor rank is higher, the state-of-the-art methods are the the normal form (NF) method \cite{telen2021normal}
and the method by Domanov and De Lathauwer (DDL) \cite{Domanov2017}. 
These are algebraic methods with theoretical guarantees for finding tensor decomposition when the tensor rank is below certain theoretical bounds. 
The NF method has partial theoretical guarantees when the tensor rank falls within the Middle-Rank Case.
Although the DDL method can handle both the Middle-Rank Case and High-Rank Case, its theoretical guarantee relies on an assumption \cite[Theorem~15]{Domanov2017} and 
it is typically unclear whether this assumption would hold for a given generic tensor. 
In addition, both methods need to construct an auxiliary matrix whose size is controlled by an integer parameter $\ell$ (the NF method has two integer parameters, and one of them is usually fixed to be $1$).
A larger $\ell$ has the potential to solve tensor decompositions of higher-rank tensors, but also substantially increases the computational complexity
and memory requirements. 
For the Middle-Rank Case and High-Rank Case, those auxiliary matrices will often be much larger than the tensor size. 
For example, consider a $(n_1,n_2,n_3)=(24,7,5)$ tensor with rank $24$, DDL method needs to set $\ell=2$ with the auxiliary matrix size $17550 \times 17550$ for a successful decomposition.
For this tensor, NF method needs to set $\ell=5$ with the auxiliary matrix size $2310 \times 2310$.
Hence, both the computational complexity and the memory requirements of DDL and NF methods increase dramatically when the tensor rank and dimensions grow.
There are also homotopy methods \cite{Hauenstein2019,KUO2018homotopy}, which reformulate the tensor decomposition problem into polynomial systems and utilize homotopy techniques 
from numerical algebraic geometry to solve them. 
However, these methods are limited to small tensors due to high memory and computational costs. 
For more discussions on decomposition of symmetric tensors and Hermitian tensors, and the uniqueness of tensor decompositions, one may refer to \cite{Bernardi2013,BERNARDI2020,bhaskara2014smoothed,BRACHAT2010,Chiantini_2017,DomanovLathauwer2013,DomanovLathauwer2015,larsen2020practical,lovitz2021generalization,nie2015generating,nie2020hermitian,RHODES20101818}.

\subsection{Contributions}

To address the computational and memory limitations mentioned previously, 
we propose a new two-stage optimization algorithm for solving the tensor decomposition problem in the Middle-Rank Case. 
For the Low-Rank Case, our method is generally the same as the generating polynomial method given in \cite{NWZ23}.

Given a Middle-Rank Case tensor $\mathcal{F}\in \mathbb{C}^{n_1\times n_2\times n_3}$ with $n_1\ge n_2\ge n_3$ and rank $n_2 < r \le n_1$, 
our algorithm aims to find a tensor decomposition as in \eqref{sum_pure}. 
The high-level framework of our algorithm can be outlined as follows:

\smallskip

\begin{itemize}
    \item[Step~1.] Preprocessing the tensor $\mathcal{F}$  to get a reduced tensor $\mathcal{T}$ with its slices $T_2,\cdots,T_{n_3}$. (See Section~3.1.)

\smallskip
    \item[Step~2.] Try to find all the generalized left common eigenvectors of $T_2,\cdots,T_{n_3}$ by solving the first optimization problem.
    If all the generalized common eigenvectors are found, a linear least squares problem is applied to get the tensor decomposition; 
    Otherwise, go to Step~3. (See Section~3.2.)

\smallskip
    \item[Step~3.] Solve the second optimization problem as in \eqref{eq:2ndOp} with the incomplete set of generalized common eigenvectors obtained from Step~2. 
    Then, solve a linear least squares problem to get the tensor decomposition. (See Section~3.3.)
\end{itemize}

\smallskip
In the ideal case, our method could successfully find the tensor decomposition after Step~2.
Otherwise, the method will construct and solve a second optimization problem in Step~3 based on the generating polynomials of the reduced tensor $\mathcal{T}$ 
and the partial left common eigenvectors identified in Step~2. 
The tensor decomposition can then be obtained by solving a linear least squares problem.
More details on these steps will be explained in Section~\ref{two-stage-optimization}, Algorithms \ref{alg:findr1} and \ref{alg:findr1:step2}.

This paper is organized as follows. Section~\ref{sc:Preliminary} introduces the notation,
reviews the generating polynomials for tensors and presents some preliminary results.
In Section~\ref{two-stage-optimization}, we describe how to reformulate the tensor decomposition problem into the first and second optimization problems and provide our optimization-based algorithms.
Section~\ref{sc:example} presents numerical experiments by  comparing our algorithm with other widely used and state-of-the-art algorithms for tensor decomposition.

\section{Notation and Preliminary Results}
\label{sc:Preliminary}
\subsection*{Notation}
We use the symbol $\mathbb{R}$ (resp., $\mathbb{C}$) to
denote the set of real numbers (resp., complex numbers). 
Curl letters (e.g., $\mathcal{F}$) denote tensors, $\mathcal{F}_{i_1,\cdots,i_m}$ denotes the $(i_1,\cdots,i_m)$-th entry of the tensor $\mathcal{F}$.
Uppercase letters (e.g., $A$) denote matrices,
$A_{ij}$ denotes the $(i,j)$-th entry of the matrix $A$.
Lowercase letters (e.g., ${v}$) denote column vectors, ${v}_i$ is its $i$-th entry and 
$\mbox{diag}({v})$ denotes the square diagonal matrix whose diagonal entries are ${v}$.
for a vector $v$, ${v}_{s:t}$ denotes the subvector $(v_s,\ldots,v_t)^\top$.
For a matrix $A$, $A_{:,j}$ and $A_{i,:}$ denote its $j$-th column and $i$-th row, respectively.
Similar subscript notations are used for tensors.
For a complex matrix $A$, $A^\top$ denotes its transpose. 
The $\text{vec}(A)$ denotes the column vector obtained by sequentially stacking the columns of $A$ together.
The $\text{null}(A),\text{col}(A),\mathrm{row}(A)$ denote the null space, column space, and row space of $A$, respectively,
and $I_r$ is the $r$ by $r$ identity matrix.
We denote the Kronecker product by $\boxtimes$ and the outer product by $\otimes$. 
They are mathematically equivalent but have different output shapes. 
For matrices $A=\begin{pmatrix}
    {a}_1,{a}_2,\cdots,{a}_n
\end{pmatrix} \in \mathbb{C}^{m \times n} $ and $B=\begin{pmatrix}
    {b}_1,{b}_2,\cdots,{b}_n
\end{pmatrix}\in \mathbb{C}^{p \times n}$, the Kronecker product is
\begin{equation*}
A \boxtimes B=\begin{pmatrix}
    A_{1,1}B & \dots & A_{1,n}B\\
    \vdots & \ddots & \vdots\\
    A_{m,1}B & \dots & A_{m,n}B\\
\end{pmatrix}\in \mathbb{C}^{mp \times n^2},
\end{equation*}
and the reverse order Khatri-Rao product is
\begin{equation*}
A \odot B \coloneqq \begin{pmatrix}
    {b}_1 \boxtimes {a}_1, {b}_2 \boxtimes  {a}_2,\cdots, {b}_n \boxtimes  {a}_n
\end{pmatrix} \in \mathbb{C}^{mp \times n}.
\end{equation*}
For the tensor $\mathcal{F}$ with decomposition $\mathcal{F} = \sum_{i=1}^r {u}^{i,1}\otimes \cdots \otimes {u}^{i,m}$, we denote the decomposition matrices of $\mathcal{F}$ by 
\begin{align}
	U^{(j)}\coloneqq \begin{pmatrix}
	    {u}^{1,j},{u}^{2,j},\cdots,{u}^{r,j}
	\end{pmatrix}, ~~ j \in \{1,2,\cdots,m\}.
\end{align}
For convenience, we also write the decomposition as 
$$
\mathcal{F}=U^{(1)} \circ U^{(2)}\circ \cdots\circ U^{(m)}\coloneqq \sum_{j=1}^r U^{(1)}_{:,j} \otimes U^{(2)}_{:,j} \otimes\cdots\otimes U^{(m)}_{:,j}.
 $$
For a matrix $V \in \mathbb{C}^{p \times n_t}$ with $1 \leq t \leq m$,
the matrix-tensor product
\begin{equation*}
\mathcal{A}\coloneqq   V \times_t \mathcal{F}
\end{equation*}
is a tensor in
$\mathbb{C}^{n_1\times \cdots\times n_{t-1} \times p \times n_{t+1} \times \cdots\times n_m}$
such that the $i$-th slice of $\mathcal{A}$ is
\begin{equation*}
\mathcal{A}_{i_1,\cdots,i_{t-1},:,i_{t+1},\cdots,i_m} = V \mathcal{F}_{i_1,\cdots,i_{t-1},:,i_{t+1},\cdots,i_m}.
\end{equation*}
A property for the matrix tensor product is 
\begin{equation}\label{m-v-prod}
V \times_1 (U^{(1)} \circ U^{(2)}\circ \cdots\circ U^{(m)}) = (V U^{(1)}) \circ U^{(2)}\circ \cdots\circ U^{(m)}.
\end{equation}
For a tensor $\mathcal{F}\in \mathbb{C}^{n_1 \times n_2 \times n_3}$, let $1 \leq i_1 \leq n_1$, $1 \leq i_2 \leq n_2$, $1 \leq i_3 \leq n_3$ and 
\begin{equation*}
    j \coloneqq 1+\sum_{l=1,l\neq k}^{3}(i_l-1)J_l \mathrm{~~with}~~ J_l \coloneqq \prod_{p=1,p\neq k}^{l-1} n_p, 
\end{equation*} 
then, the mode-$k$ flattening is defined as 
\begin{equation*}
    M \coloneqq \mathrm{Flatten}(\mathcal{F},k)\in \mathbb{C}^{n_k \times \frac{n_1n_2n_3}{n_k}},~\mathrm{where~}M_{i_k,j}=\mathcal{F}_{i_1,i_2,i_3}.
\end{equation*}
Finally, we assume throughout the paper that
the tensor dimension $(n_1, n_2 ,n_3)$ is in descending order, i.e., $n_1\geq n_2 \geq n_3$.

\subsection{Generating Polynomial and Tensor Decomposition}
\label{sub:generatingPoly}
Generating polynomials are closely related to tensor decomposition. In this section, we review the definition of generating polynomials 
and the relations of the tensor decomposition to an optimization problem using generating polynomials.

For a tensor $\mathcal{F}\in\mathbb{C}^{n_1\times n_2\times \cdots \times n_m}$ with rank $r\le n_1$, we can index the tensor by monomials
\begin{align} \label{def:labelF}
\mathcal{F}_{x_{1,i_1},\ldots, x_{m,i_m}}\coloneqq \mathcal{F}_{i_1,\ldots,i_m}.
\end{align}
Consider a subset $I \subseteq \{1,2,\cdots,m \}$, we define
\begin{equation} 
\begin{array}{rcl}
I^c  & \coloneqq & \{1,2,\cdots,m \} \backslash I,  \\
 \mathbb{M}_{I} &\coloneqq & \{\Pi_{j\in I} x_{j,i_j}|1\le i_j \le n_j\}, \\
 \mathcal{M}_I & \coloneqq & {\rm span} \{\mathbb{M}_I\},\\
 \end{array}
\end{equation}
and the bi-linear operation
$\langle \cdot, \cdot \rangle$
between $\mathcal{M}_{\{1,2,\cdots,m\}}$ and $\mathbb{C}^{n_1 \times \ldots \times n_m}$ as
\begin{align}  	\label{polyten}
\langle ~\sum_{\mathclap{\scriptstyle \mu \in \mathbb{M}_{\{1,2,\cdots,m\}}}}~c_{\mu}\mu,\mathcal{F}~ \rangle~\coloneqq~\sum_{\mathclap{\scriptstyle \mu \in \mathbb{M}_{\{1,2,\cdots,m\}}}} ~c_{\mu}\mathcal{F}_{\mu},
\end{align}
where for monomial $\mu$,  $c_{\mu} \in \mathbb{C}$ is a scalar and $\mathcal{F}_\mu$ is the element of $\mathcal{F}$ labelled as in (\ref{def:labelF}).
By denoting,
\begin{align}
J\coloneqq  \{(i,j,k)|1\le i\le r, 2\le j\le m,2\le k\le n_j\},
\end{align}
we can have the following definitions on generating polynomials and generating matrices of a tensor $\mathcal{F}$ .
\begin{definition}[\cite{NieLR14,NWZ22}] \label{def:generatingmatrix}
For a subset $I \subseteq \{1,2,\cdots,m\}$ and a tensor
$\mathcal{F} \in \mathbb{C}^{n_1 \times \cdots \times n_m}$ with rank $r\le n_1$,
a polynomial $p \in \mathcal{M}_I$
is called a \textbf{generating polynomial} for $\mathcal{F}$ if
\begin{equation}  		\label{inner_p}
\langle pq,\mathcal{F} \rangle =0 \quad
\mbox{for all} \, \, q \in \mathbb{M}_{I^c} .
\end{equation}
Furthermore, the matrix $G \in \mathbb{C}^{r \times \lvert J \rvert }$ is called a \textbf{generating matrix} of $\mathcal{F}$ if the following equation
\begin{equation}  	\label{linear_eq_g}
    \sum_{\ell=1}^{r} G(\ell,\tau)\mathcal{F}_{x_{1,\ell}x_{j,1} \cdot \mu}
    =\mathcal{F}_{x_{1,i}  x_{j,k}  \cdot \mu} 
    \end{equation}
holds for all $\mu \in \mathbb{M}_{\{1,j\}^c}$ and $\tau=(i,j,k) \in J$.
\end{definition}

For $2\le j\le m$ and $2\le k \le n_j$, we define the $r$ by $r$ sub-matrix $M^{j,k}[G]$ of the generating matrix $G$ as
\begin{equation}\label{def:subgenerating}
M^{j,k}[G] \coloneqq \begin{pmatrix}
	G(1,(1,j,k)) & G(2,(1,j,k)) & \dots & G(r,(1,j,k)) \\
	G(1,(2,j,k)) & G(2,(2,j,k)) & \dots & G(r,(2,j,k)) \\
	\vdots & \vdots & \ddots & \vdots \\
	G(1,(r,j,k)) & G(2,(r,j,k)) & \dots & G(r,(r,j,k)) \\
\end{pmatrix},
\end{equation}
and the matrices
\begin{equation}  \label{def:A_b}
\begin{cases}
A[\mathcal{F},j]  &  \coloneqq  
\Big( \mathcal{F}_{x_{1,\ell}\cdot x_{j,1}\cdot \mu}
     \Big)_{\mu \in \mathbb{M}_{\{1,j\}^c}, 1\leq \ell \leq r} \in \mathbb{C}^{m_j^c \times r}, \\
 B[\mathcal{F},j,k] & \coloneqq  
     \Big(\mathcal{F}_{x_{1,\ell} \cdot x_{j,k}\cdot \mu}
      \Big)_{\mu \in \mathbb{M}_{\{1,j\}^c}, 1\leq \ell \leq r} \in \mathbb{C}^{m_j^c \times r},
\end{cases}
\end{equation}
where  $m^c_j=\frac{n_1 n_2 \dots n_m }{n_1 n_j}$.
Because $G$ is a generating matrix, by \eqref{linear_eq_g}, for all $2\le j\le m$ and $2\le k\le n_j$,
 those matrices follow the linear relation
\begin{equation} \label{linear_eq}
A[\mathcal{F},j] (M^{j,k}[G])^\top \, = \, B[\mathcal{F},j,k].
\end{equation}
Besides \eqref{linear_eq}, \zhc{the matrix} $M^{j,k}[G]$ also has the following property.
\begin{theorem}[\cite{NieLR14,NWZ22}]  \label{thm:decom_to_Mjk}
Suppose $\mathcal{F} =\sum_{i=1}^r {u}^{i,1} \otimes \cdots \otimes {u}^{i,m}$
for vectors ${u}^{i,j} \in \mathbb{C}^{n_j}$.
If $r \leq n_1$, ${ {u}^{i,2}_1 \cdots {u}^{i,m}_1 } \neq 0$ for $i=1,\cdots,r$,
and the first $r$ rows of the first decomposing matrix
\begin{equation*}
U^{(1)} \coloneqq \begin{pmatrix}
    {u}^{1,1} \, \, \cdots \,\,  {u}^{r,1} 
\end{pmatrix}
\end{equation*}
are linearly independent,
then there exists a generating matrix $G$ satisfying \eqref{linear_eq}
and for all $2\le j \le m$,
$2\le k \le n_j$ and $1\le i \le r$, it holds that
\begin{align}  	\label{Mjk[G]:eigeqn}
{ M^{j,k}[G]\cdot {u}^{i,1}_{1:r}=
{u}^{i,j}_k\cdot {u}^{i,1}_{1:r}. }
\end{align}
\end{theorem}

For a generic tensor $\mathcal{F}$ with rank $r \leq n_1$, Theorem~\ref{thm:decom_to_Mjk} along with \cite[Theorem 4.1]{NWZ23} implies that
 there is an equivalence relation between a tensor decomposition and a generating matrix $G$ such that the $M^{j,k}[G]$'s, 
$2\le j\le m$ and $2\le k\le n_j$, are simultaneously diagonalizable. 
This characterizes how the generating matrices are related to the tensor decomposition.
In conclusion, when the tensor rank   $r \leq n_1$, the generating polynomials in the generic case construct a bijective relationship 
between the tensor decomposition and the $M^{j,k}$'s satisfying both \eqref{linear_eq} and \eqref{Mjk[G]:eigeqn}.

\subsection{Generating polynomial for order-3 tensor in Middle-Rank Case}\label{section:GPorder3}
For a tensor  $ \mathcal{F} \in \mathbb{C}^{n_1 \times n_2 \times n_3} $ in the Middle-Rank-Case, i.e., $n_2 < r \le n_1$,
if its decomposition is unique, \cite{Chiantini_AnAlgorithm_2014} shows that the tensor decomposition
can be generically recovered from the tensor decomposition of the sub-tensor $ \mathcal{F}_{1:r,:,:} $.
In practice, when rank $r < n_1$, the tensor decomposition can be also recovered from 
decomposing a smaller sized core tensor generated by some tucker tensor decomposition methods such as
 HOSVD \cite{Lathauwer2000} and HOID \cite{SaibabaHOID2016}. 
The following Lemma, which can be also obtained from the GEVD point of view,
 provides a simple proof showing that in the generic case, the tensor decomposition
of $ \mathcal{F}$ can be obtained by using the tensor decomposition of its sub-tensor
 $ \mathcal{F}_{1:r,:,:}$ without the uniqueness assumption.
\begin{lemma} \label{lemma:rton1}
Let $\mathcal{F}\in \mathbb{C}^{n_1 \times n_2 \times  n_3}$ be an order-3 tensor with
 rank $n_2 < r \le n_1$. 
 Suppose $\mathcal{F}_{1:r,:,:}=\hat{U}^{(1)} \circ \hat{U}^{(2)} \circ \hat{U}^{(3)}$
 and let  $A^1= \hat{U}^{(2)} \odot \hat{U}^{(3)}$ and $B^1=\mathrm{Flatten}(\mathcal{F},1)^\top$.
Then, in the generic case, the linear system  $A^1X=B^1$ has a least squares solution
 $\tilde{U}^{(1)}$ and  $\mathcal{F}=\tilde{U}^{(1)} \circ \hat{U}^{(2)}  \circ \hat{U}^{(3)}$.
\end{lemma}
\begin{proof}
Since $\mathcal{F}$ has rank $n_2 < r \le n_1$, we have 
$\mathcal{F}={U}^{(1)} \circ {U}^{(2)} \circ {U}^{(3)}$ for some decomposition matrices 
${U}^{(j)} \in \mathbb{C}^{n_j \times r}$, $j = 1, 2, 3$, and in the generic case
${U}^{(1)}_{1:r,:}$ is nonsingular.
So, denoting $W_1\coloneqq  {U}^{(1)}({U}^{(1)}_{1:r,:})^{-1}$, we have from
\begin{equation*}
\mathcal{F}_{1:r,:,:}={U}^{(1)}_{1:r,:} \circ {U}^{(2)} \circ {U}^{(3)} =\hat{U}^{(1)} \circ \hat{U}^{(2)} \circ \hat{U}^{(3)},
\end{equation*}
property \eqref{m-v-prod} and $\mathcal{F}={U}^{(1)} \circ {U}^{(2)} \circ {U}^{(3)}$ that
 imply
\begin{equation*}
   \mathcal{F} =  W_1 \times_1  \mathcal{F}_{1:r,:,:}=W_1 \hat{U}^{(1)} \circ \hat{U}^{(2)} \circ \hat{U}^{(3)}. 
\end{equation*}
Hence, the linear system
\begin{equation}\label{eq:leastsquares}
    \left(\hat{U}^{(2)} \odot \hat{U}^{(3)}  \right) X= \mathrm{Flatten}(\mathcal{F},1)^\top
\end{equation}
has a solution $X= (W_1 \hat{U}^{(1)} )^\top$. 
Therefore,  \eqref{eq:leastsquares} has a least square solution
$\tilde{U}^{(1)}$ with zero residue. And, because of \eqref{eq:leastsquares},
we have 
$\mathcal{F}=\tilde{U}^{(1)} \circ \hat{U}^{(2)} \circ \hat{U}^{(3)}$.
\end{proof}

\section{Equivalent optimization reformulation and two-stage algorithm}
\label{two-stage-optimization}
In this section,
we derive equivalent optimization formulations of tensor decomposition 
and present our two-stage optimization algorithm. 
We begin by obtaining a reduced tensor $\mathcal{T}$ through preprocessing of a generic tensor $\mathcal{F}$.

\subsection{Preprocessing for reduced tensor $\mathcal{T}$}
For a generic tensor $\mathcal{F}\in \mathbb{C}^{n_1 \times n_2 \times n_3}$ with 
rank $n_2 < r \leq n_1$, suppose $\mathcal{F}$ has tensor decomposition 
 $\mathcal{F}=U^{(1)} \circ U^{(2)} \circ U^{(3)}$.
We would like to obtain a reduced tensor  $\mathcal{T}$ by preprocessing 
the tensor $\mathcal{F}$.
 
First, in the generic case, we would have $U^{(1)}_{1:r,:}$ is nonsingular, 
$U^{(2)}$ has full row rank and $U^{(3)}_{1,s} \ne 0$ for all $s=1, \ldots, r$.
For any $\lambda_s \ne 0$, $s=1, \ldots, r$, we have the observation 
\begin{equation*} U^{(1)} \circ U^{(2)} \circ U^{(3)}=U^{(1)} \circ (U^{(2)} \mathrm{diag}(\begin{pmatrix}
    \lambda_1,\cdots,\lambda_r
\end{pmatrix})) \circ (U^{(3)} \mathrm{diag}(\begin{pmatrix}
    1/\lambda_1,\cdots,1/\lambda_r
\end{pmatrix})).
\end{equation*}
Hence, without loss of generality, we can assume 
that $U^{(3)}_{1,s}=1$ for $s\in \{1,\cdots,r\}$. 
Therefore, we have
\begin{equation}\label{eq:f1fullrank}
\mathcal{F}_{1:r,:,1}=U^{(1)}_{1:r,:}\mathrm{diag}(U^{(3)}_{1,:})(U^{(2)})^\top=U^{(1)}_{1:r,:}(U^{(2)})^\top \in \mathbb{C}^{n_1 \times n_2},
\end{equation}
which has full column rank.
Therefore, there exists a matrix $C \in \mathbb{C}^{r \times (r-n_2)}$ such that 
\[
\hat{F}\coloneqq \begin{pmatrix}
    \mathcal{F}_{1:r,:,1} & C \\
\end{pmatrix}\in \mathbb{C}^{r \times r}
\]
is nonsingular.
So, denoting $P\coloneqq \hat{F}^{-1}$, we will have $P$ is nonsingular and 
\begin{equation}\label{eq:pre-processing}
    P \mathcal{F}_{1:r,:,1}=(I_r)_{:,1:n_2}.
\end{equation}
Now let the tensor $\mathcal{T}$ be obtained by matrix-tensor product 
on the tensor $\mathcal{F}_{1:r,:,:}$ as
\begin{align}\label{pre-processing}
    \mathcal{T}\coloneqq P \times_1 \mathcal{F}_{1:r,:,:}.
\end{align}
Then, from tensor decomposition of $\mathcal{F}$ and the property \eqref{m-v-prod},
we have 
\begin{equation}\label{T-decomp}
 \mathcal{T} = \hat{U}^{(1)} \circ U^{(2)} \circ U^{(3)}, 
\end{equation}
where $\hat{U}^{(1)}\coloneqq PU^{(1)}_{1:r,:}$.
Defining $T_k = \mathcal{T}_{:,:,k}$,  $1 \le k \le n_3$,
 it then follows from our construction that 
 $T_1=P\mathcal{F}_{1:r,:,1}=(I_r)_{:,1:n_2}$.
Our first-stage optimization algorithm relates to finding the generalized left common eigenvectors of $T_2, T_3, \ldots, T_{n_s}.$
(See Definition~\ref{def-gen-eigen}).
 
We now consider the linear systems \eqref{linear_eq} with $j=3$ and 
the tensor $\mathcal{F}$ being replaced by \zhc{the reduced tensor} $\mathcal{T}$.
Then, for all  $2 \le k \le n_3$,  denoting  $M^{3,k}[G]$ as $M^{3,k}$,
the linear systems in \eqref{linear_eq} with $j=3$ can be rewritten as
\begin{equation}\label{Mk-Tk}
    M^{3,k} T_1 = T_k.
\end{equation}
However, when  $r>n_2$, the matrices $\{ M^{3,k} \}_{k=2}^{n_3}$ cannot be fully determined by the linear systems \eqref{Mk-Tk}.
Fortunately, by Theorem~\ref{thm:decom_to_Mjk}, in addition to equations \eqref{Mk-Tk}, in the generic case
$ \{ M^{3,k} \}_{k=2}^{n_3}$ are simultaneously diagonalizable, and hence, must mutually commute, that is,
 for all $2 \le i < j \le n_3$, we have
\begin{equation} \label{eq:commute}
    M^{3,i} M^{3,j} = M^{3,j}M^{3,i}.
\end{equation}
Now, for all $2 \le k \le n_3$, by defining $P_k \coloneqq M^{3,k}_{:,n_2+1:r}$ and plugging in $T_1 = \mathcal{T}_{:,:,1}=(I_r)_{:,1:n_2}$ to \eqref{Mk-Tk},
we would have
\begin{align}\label{eq:def:P}
    M^{3,k} = \begin{pmatrix}
    T_k & P_k
\end{pmatrix}.
\end{align}
Since $\{T_k \}_{k=2}^{n_3}$ are known, finding $ \{ M^{3,k} \}_{k=2}^{n_3}$ turns out 
to be finding  $ \{ P_k \}_{k=2}^{n_3}$.
With \eqref{eq:def:P}, for all $2 \le i < j \le n_3$, the commuting equations \eqref{eq:commute}
can be rewritten as
\begin{eqnarray*}
  0  &=&M^{3,i}M^{3,j} - M^{3,j}M^{3,i}\\
    &=&  \begin{pmatrix} T_i & P_i \end{pmatrix} \begin{pmatrix} T_j & P_j \end{pmatrix} -
          \begin{pmatrix} T_j & P_j \end{pmatrix} \begin{pmatrix} T_i & P_i \end{pmatrix} .
\end{eqnarray*}
This gives the following linear and quadratic equations on unknowns $\{P_k\}_{k=2}^{n_3}$:
\begin{equation} \label{eq:linearT}
\begin{pmatrix}
    T_i & P_i\\
\end{pmatrix}T_j-
\begin{pmatrix}
    T_j & P_j\\
\end{pmatrix}T_i=0,
\end{equation}
and 
\begin{equation}\label{eq:quadT}
\begin{pmatrix}
    T_i & P_i\\
\end{pmatrix}P_j-
\begin{pmatrix}
    T_j & P_j\\
\end{pmatrix}P_i=0.
\end{equation}
We would use these linear equations \eqref{eq:linearT}
and nonlinear equations \eqref{eq:quadT} for designing our second-stage optimization 
algorithm.
\subsection{The first-stage optimization algorithm}
In this subsection, we propose the first reformulated optimization problem that is equivalent to the tensor decomposition problem for generic tensors with rank $r \leq n_1$.
We focus on the Middle-Rank Case with $n_2 < r \le n_1$.
For the Low-Rank Case with $r \leq n_2$, our method would essentially in spirit similar to
to the generalized eigenvalue decomposition (GEVD) method \cite{Leurgans1993,Sanchez1990}.

Recall the decomposition $\mathcal{T} = \hat{U}^{(1)} \circ U^{(2)} \circ U^{(3)}$
defined in \eqref{T-decomp}. 
Our first goal is to find the inverse of the first decomposition matrix $ \hat{U}^{(1)} $ 
of $\mathcal{T}$.
Denoting $S\coloneqq (\hat{U}^{(1)})^{-1}$ and recalling the definition of
 ${T}_{k}={\mathcal{T}}_{:,:,k}$, for all $k=1, \ldots, n_3$, we have 
\begin{align*}
    ST_k=(\hat{U}^{(1)})^{-1}\hat{U}^{(1)}\mathrm{diag}(U^{(3)}_{k,:})(U^{(2)})^\top=\mathrm{diag}(U^{(3)}_{k,:})(U^{(2)})^\top.
\end{align*} 
Then, it follows from  $T_1=(I_r)_{:,1:n_2} $ and our assumption
$\mathrm{diag}(U^{(3)}_{1,:}) = I_r$ that $S_{:,1:n_2} =  (U^{(2)})^\top$.
Moreover, denoting  $D_k\coloneqq \mathrm{diag}(U^{(3)}_{k,:})$,
we have
\begin{equation} \label{S-Tk}
S T_k= D_k (U^{(2)})^\top =D_kS_{:,1:n_2}
\end{equation}
 for all $k=1, \ldots, n_3$.
Note that $T_k$ in \eqref{S-Tk} is a $r$ by $n_2$ matrix instead of a square matrix.
So, the rows of $S$ can be considered as the generalized left eigenvectors of $T_k$.
Motivated by this observation, we propose the following definition of
the generalized left common eigenmatrix and eigenvectors.
\begin{definition} \label{def-gen-eigen}
    For a set of matrices $A_1,\cdots,A_d \in \mathbb{C}^{m \times n}$ with $m \geq n$, a full rank matrix $S \in \mathbb{C}^{m \times m}$ is called the generalized left common eigenmatrix of $A_1,\cdots,A_d$, if it satisfies
    \begin{equation} \label{eq:gceig}
            S A_k  = D_k S_{:,1:n} ~~\mathrm{for~}1\leq k \leq d, 
    \end{equation}
    where $D_k \in \mathbb{C}^{m \times m}$ is a diagonal matrix. 
    Then, for all $i =1, \ldots,m$, $s^i \coloneqq S_{i,:}$ is called a generalized left common eigenvector of 
$A_1,\cdots,A_d$, and $\lambda_{i,k} \coloneqq (D_k)_{i,i}$s is called the generalized left common eigenvalue of $A_k$ associated with $s^i$.
\end{definition}

Since $T_1=(I_r)_{:,1:n_2} $, $S T_1=S_{:,1:n_2}$ naturally holds for all $S$. 
So, our goal is to find the generalized left common eigenmatrix $S$ of the reduced tensor 
slices $T_2, T_3,\cdots,T_{n_3}\}$, i.e., find $S \in \mathbb{C}^{r \times r}$ and
$\lambda_{i,k} \in \mathbb{C}$ such that
\begin{equation*}
S_{i,:} T_k  = \lambda_{i,k} S_{i,1:n_2} ~~\mathrm{for~all~}1 \leq i\leq r \mbox{ and } 2 \leq k \leq n_3.
\end{equation*} 
Of course, under different scenarios, the generalized left common eigenmatrix and eigenvectors may not exist,
and even if it exists, it may not be unique. However, the following theorem shows that in 
the Middle-Rank Case, the generalized left common eigenmatrix of the reduced tensor slices has a bijective relationship with the tensor decomposition,
which can be utilized to find the tensor decomposition.
\begin{theorem} \label{thm:SequDecom}
Let $\mathcal{F}\in \mathbb{C}^{n_1 \times n_2 \times  n_3}$ be an order-3 tensor with
 rank $n_2 < r \le n_1$. 
 Suppose $\mathcal{T}$ is the reduced tensor of $\mathcal{F}$ given in \eqref{eq:pre-processing}
 with  ${T}_{k}={\mathcal{T}}_{:,:,k}$, $k=1, \ldots, n_3$.
In the generic case, we have
    \begin{enumerate}[(i)]
        \item for each nonsingular generalized left common eigenmatrix $S$ of $T_2,\cdots,T_{n_3}$,
         $\mathcal{F}$ has a tensor decomposition given in \eqref{eq:findTDbygceig};
        \item  for each tensor decomposition, there is a nonsingular generalized left 
        common eigenmatrix $S$ of $T_2,\cdots,T_{n_3}$.
    \end{enumerate}
\end{theorem}
\begin{proof}
We first prove $(i)$. Suppose $S$ is a nonsingular generalized left common eigenmatrix of 
$T_2,\cdots,T_{n_3}$, that is 
\begin{equation}\label{S-Tk-1}
    S T_k = D_k S_{:,1:n_2}~~\mathrm{for~} 2\leq k \leq n_3. 
\end{equation}
Let $\lambda_{i,k}=(D_k)_{i,i}$ for all $1\leq i \leq r$ and $2\leq k \leq n_3$, and
let $\hat{\mathcal{T}} = \hat{U}^{(1)} \circ U^{(2)} \circ U^{(3)}$, where
\begin{align}\label{eq:decombyS}
    \hat{U}^{(1)} = S^{-1},~U^{(2)} = (S_{:,1:n_2})^\top \mbox{ and } U^{(3)} = \begin{pmatrix}
        1 & 1 & \dots & 1 \\
        \lambda_{1,2} & \lambda_{2,2} & \dots & \lambda_{r,2} \\
        \vdots  &\vdots  &\vdots  &\vdots  \\
         \lambda_{1,n_3} & \lambda_{2,n_3} & \dots & \lambda_{r,n_3} \\
    \end{pmatrix}.
\end{align}
Then, by the construction of $\hat{\mathcal{T}}$ and \eqref{S-Tk-1}, 
\begin{align*}
\hat{\mathcal{T}}_{:,:,k}=\hat{U}^{(1)}\mathrm{diag}(U^{(3)}_{k,:})(U^{(2)})^\top
=S^{-1}D_{k}S_{:,1:n_2} = S^{-1} (S T_k) =T_k.
\end{align*}
Hence, we have $\hat{\mathcal{T}}=\mathcal{T}$. 
Then, it follows from \eqref{eq:pre-processing} that
\begin{align*}
 \mathcal{F}_{1:r,:,:}=P^{-1}\times_1 \mathcal{T}= P^{-1}\times_1 \hat{\mathcal{T}} =
 P^{-1} \hat{U}^{(1)} \circ U^{(2)} \circ U^{(3)}.
\end{align*}
This gives a tensor decomposition for $\mathcal{F}_{1:r,:,:}$.
Then, by Lemma \ref{lemma:rton1}, the linear least squares system $AX=B$
has a solution, denoted as ${U}^{(1)}$, where $A=U^{(2)} \odot U^{(3)}$, 
$B=\mathrm{Flatten}(\mathcal{F},1)^\top$, and we have a tensor decomposition of 
$\mathcal{F}$ as 
\begin{align} \label{eq:findTDbygceig}
    \mathcal{F}={U}^{(1)} \circ U^{(2)} \circ U^{(3)}.
\end{align}

We now prove $(ii)$. 
This essentially follows from the previous discussion on the motivations of the Definition~\ref{def-gen-eigen}.
Since $\mathcal{T}$ is the reduced tensor of $\mathcal{F}$ given in \eqref{eq:pre-processing}, we have \eqref{T-decomp} holds. That is 
$\mathcal{T} = \hat{U}^{(1)} \circ U^{(2)} \circ U^{(3)}$, 
where $\hat{U}^{(1)} = PU^{(1)}_{1:r,:}$, $P$ is given in \eqref{eq:pre-processing}
and $ U^{(i)} $, $i=1,2,3$, are matrices such that 
$\mathcal{F} = U^{(1)} \circ U^{(2)} \circ U^{(3)}$.
Let $S = (\hat{U}^{(1)})^{-1}$.
Then, for all $1 \leq k \leq n_3$, we have
\begin{equation}\label{TkTkTk}
T_k= PU^{(1)}_{1:r,:} \text{diag}(U^{(3)}_{k,:})({U}^{(2)})^\top=S^{-1} \text{diag}(U^{(3)}_{k,:})({U}^{(2)})^\top.
\end{equation}
It then follows from  $T_1=(I_r)_{:,1:n_2} $,  
$\mathrm{diag}(U^{(3)}_{1,:}) = I_r$ and \eqref{TkTkTk} that $S_{:,1:n_2} =  (U^{(2)})^\top$.
Hence, by  \eqref{TkTkTk}, for $2 \leq k \leq n_3$ we have
\begin{align}
    S T_k = \text{diag}(U^{(3)}_{k,:})({U}^{(2)})^\top= \text{diag}(U^{(3)}_{k,:}) S_{:,1:n_2}.
\end{align}
Therefore, $S$ is a nonsingular generalized left common eigenmatrix of 
$T_2,\cdots,T_{n_3}$.
\end{proof}

Theorem~\ref{thm:SequDecom} shows that when the tensor rank belongs $n_2 < r \le n_1$,
 in the generic case, the reduced tensor slices would have a generalized left common eigenmatrix $S$,
 which can be used to construct the tensor decomposition.
In the following, we propose an optimization-based approach to find the rows of $S$ sequentially. 
First, given an unitary matrix $Q \in \mathbb{C}^{r \times r}$, 
for any $x \in \mathbb{C}^{r-1}$ denoting  $\bar{x}=Q\begin{pmatrix}  x^\top & 1 \\    \end{pmatrix}^\top$,
we define the function $f_{Q}(x)$ with domain
 $\Omega \coloneqq \{x \in \mathbb{C}^{r-1}: \bar{x}_{1:n_2} \ne 0 \}$  as 
\begin{align}
    f_{Q}(x) \coloneqq \text{Vec}\left( \big(I_{n_2}-\frac{\bar{x}_{1:n_2} \bar{x}^\top_{1:n_2}}{\bar{x}_{1:n_2}^\top \bar{x}_{1:n_2}} \big) 
\big(\bar{x}^\top \times_1 \mathcal{T} \big)\right).
\end{align}
Here, $ \bar{x}_{1:n_2} \bar{x}^\top_{1:n_2}/(\bar{x}_{1:n_2}^\top \bar{x}_{1:n_2})$ is a projection matrix that projects a vector $v \in \mathbb{C}^{n_2}$ into the column space of $\bar{x}_{1:n_2}$.
By the property of projection matrices, one can verify that $f_{Q}(x)=0$ if and only if 
all  $k =1, \ldots,n_3$, we have
\begin{equation}\label{fd-rowS}
 \big(\bar{x}^\top \times_1 \mathcal{T} \big)_{:,k}=\big(\mathcal{T}_{:,:,k} \big)^\top \bar{x}   =\lambda_{1,k}  \bar{x}_{1:n_2}
\end{equation} 
for some $\lambda_{1,k} \in \mathbb{C}$.
Additionally, denoting $Z= I_{n_2}-\bar{x}_{1:n_2} \bar{x}_{1:n_2}^\top/(\bar{x}_{1:n_2}^\top \bar{x}_{1:n_2})$ and $e_i$ be the $i$-th \zhc{coordinate basis in $\mathbb{C}^{n_2}$}, 
one can derive the Jacobian matrix of $f_{Q}(x)$: 
\[
    J_{f_{Q}} = \left[\mathrm{Flatten} \big(  ( \bar{x}^\top \times_1 \mathcal{T})^\top \times_2 \frac{\partial Z}{\partial \bar{x}},3 \big)+ 
\mathrm{Flatten} \big(Z \times_2 \mathcal{T},1 \big) \right]^\top \frac{\partial \bar{x}}{\partial x},
\]    
where
\begin{align*}
    \frac{\partial Z}{\partial \bar{x}_{1:n_2}}&=\frac{-\sum_{i=1}^{n_2}( e_i \otimes \bar{x}_{1:n_2} \otimes e_i+\bar{x}_{1:n_2} \otimes e_i \otimes e_i)}{\bar{x}_{1:n_2}^\top \bar{x}_{1:n_2}}+\frac{2\bar{x}_{1:n_2} \otimes \bar{x}_{1:n_2} \otimes \bar{x}_{1:n_2}}{(\bar{x}_{1:n_2}^\top \bar{x}_{1:n_2})^2},\\
    \frac{\partial Z}{\partial \bar{x}_{n_2+1:r}}&= 0 \quad \mathrm{~and~} \quad \frac{\partial \bar{x}}{\partial x}=Q_{:,1:r-1}.
\end{align*}
\begin{remark}
    For a tensor $\mathcal{F}= U^{(1)}\circ U^{(2)}\circ U^{(3)}$ with rank $r$, we know all the columns of $U^{(2)}$ are nonzero vectors. 
    Based on \eqref{eq:decombyS}, we have $S_{i,1:n_2}=U^{(2)}_{:,i}$. 
    Therefore, for each generalized left common eigenvector $s^i$ with $(s^i)_{1:n_2}=U^{(2)}_{:,i} \neq 0$, 
    the global minimizer such that $\bar{x}=s^i$ will be in $\Omega$.
    This implies that solving $f_{Q}(x)=0$ on $\Omega$ will be generically sufficient to find all the generalized left common eigenvectors 
    in $S$ of Theorem \ref{thm:SequDecom} $(ii)$.
\end{remark}

To find the first row of $S$, we start with  a randomly generated unitary matrix $Q_1 \in \mathbb{C}^{r \times r}$
and formulate the optimization problem:
\begin{align}\label{fun:fmin}
    \min_{x\in \Omega} ~\lVert f_{Q_1}(x)\rVert^2_2.
\end{align}
If we can find the global minimizer $x^1$ of \eqref{fun:fmin} such
that $ f_{Q_1}(x^1)=0$, we let
$(s^1)^\top \coloneqq \begin{pmatrix} (x^1)^\top & 1 \\ \end{pmatrix} Q_1^\top$ be the first row of $S$.
Now, suppose we have already found the first $p-1$ rows of $S$ for some $ 1 < p \le r$.  
Let $S^{p-1} \coloneqq (s^1, \ldots, s^{p-1})^\top$.
To determine the $p$-th row of $S$, which must be linearly independent to the first $p-1$ rows,
we perform the QR decomposition of $(S^{p-1})^\top$, i.e., we find
$(S^{p-1})^\top = Q_p R_p$ for some unitary matrix $Q_p \in \mathbb{C}^{r \times r}$
and $R_p \in \mathbb{C}^{r \times (p-1)}$ with $(R_p)_{1:p-1,:} $ being a nonsingular upper triangular matrix.
Then, we formulate the optimization problem: 
\begin{align} \label{fun:fmink}
    \min_{x\in \Omega} ~\lVert f_{Q_p}(x)\rVert^2_2.
\end{align}
If we can find the global minimizer $x^p$ of \eqref{fun:fmink} such
that $ f_{Q_p}(x^p)=0$, we let
$(s^p)^\top \coloneqq \begin{pmatrix} (x^p)^\top & 1 \\ \end{pmatrix} Q_p^\top$ be the first $p$-th row of $S$; otherwise, we stop the process.
Furthermore, the next lemma demonstrates that the $p$ rows sequentially generated by the above process are 
linearly independent. 
\begin{lemma} 
Given $1 < p \le r$, suppose $s^i$, $i=1, \ldots, p$ are obtained by 
setting $s^i=Q_{i}\begin{pmatrix}
        (x^i)^\top & 1 \\
    \end{pmatrix}^\top$, where $x^i$ is the minimizer of 
 $ \min_{x\in \Omega} ~\lVert f_{Q_i}(x)\rVert^2_2$
 and $Q_i$ is constructed using the above process.
Then, $s^p \notin \mathrm{span}(s^1,\cdots,s^{p-1})$.
\end{lemma}
\begin{proof}
    As $Q_p R_p$ is the QR decomposition of $ S_{p-1}^\top =  (s^1, \ldots, s^{p-1})^\top$, 
    we have
    \begin{equation}\label{Qk-k-1}
        \mathrm{col}((Q_{p})_{:,1:p-1})=\mathrm{span}(s^1, s^2, \dots ,s^{p-1}).
    \end{equation}
  On the other hand, by the process of obtaining $s^k$, we have
    \begin{equation*}
    s^p=Q_{k}\begin{pmatrix}
        (x^p)^\top & 1 \\
    \end{pmatrix}^\top=(Q_{p})_{:,1:p-1} (x^p)_{1:p-1}+(Q_{p})_{:,p:r} \begin{pmatrix}
        (x^p)_{p:r-1} \\
        1 \\
    \end{pmatrix}.
    \end{equation*}
    Then, because $Q_p \in \mathbb{C}^{r \times r} $ is an unitary matrix, \eqref{Qk-k-1}  
    and $\begin{pmatrix}
        (x^p))_{p:r-1}^\top & 1 \\
       \end{pmatrix} \ne 0$,
     we have $s^p \notin \mathrm{span}(s^1,\cdots,s^{p-1})$.
\end{proof}

By Theorem \ref{thm:SequDecom}, if we are able to determine all the $r$ rows of the matrix $S$, i.e. find the entire matrix $S$,
a tensor decomposition for $\mathcal{F}$ can be obtained as in \eqref{eq:findTDbygceig}. 
 In this case, the algorithm for finding
 the tensor decomposition can be described as Algorithm~\ref{alg:findr1}.
However, if only the first $p$ rows $S^p$ of $S$ with $p < r$ are computed,  we cannot 
fully determine the tensor decomposition using this partial information. 
Nonetheless, the relation \eqref{eq:lineareig} involving $S^p$ will still be utilized in the second-stage optimization algorithm presented in the next section.
 
\smallskip 
\begin{alg}
\label{alg:findr1}
The First-Stage Algorithm for Tensor Decomposition
\smallskip

\noindent \textbf{Input:} The tensor $\mathcal{F}$ with rank $n_2 < r \leq n_1$.

\vskip 2pt 
\begin{itemize}

\item [Step~1]
Preprocess the tensor $\mathcal{F}$ and get the new tensor $\mathcal{T}$ as in \eqref{pre-processing}.
\smallskip

\item [Step~2]
For $k = 1, \ldots, r$, solve the optimization $ \min_{x\in \Omega} ~\lVert f_{Q_k}(x)\rVert^2_2$
sequentially and obtain $s^k$ as described in the above process.
If all the $r$ optimization problems are successfully solved, form 
the generalized left common eigenmatrix $S = (s^1, \ldots, s^r)^\top$ and continue Step~3;
otherwise, if only $p < r$ optimization problems are solved, 
form partial left common eigenmatrix $S^p = (s^1, \ldots, s^p)^\top$ and stop the algorithm.

\smallskip
\item [Step~3]
For $i = 1, \ldots, r$ and $k = 1, \ldots, n_3$,
let $\lambda_{i,k}$ be the generalized common eigenvalue of $T_k$ associated with $s^i$, 
${w}^{i,2}\coloneqq (S_{i,1:n_2})^\top $ and $w^{i,3}_k\coloneqq \lambda_{i,k}$.

\smallskip
\item [Step~4] Solve the linear system
$
    \sum_{i=1}^r  {{w}}^{i,1} \otimes {w}^{i,2} \otimes {w}^{i,3} = \mathcal{F} 
$
to get vectors $\{{w}^{i,1}\}$. 

\end{itemize}
\vskip 2pt 
\textbf{Output:} A decomposition of $\mathcal{F}$:
$  \quad \mathcal{F} = \sum_{i=1}^r {w}^{i,1} \otimes {w}^{i,2} \otimes {w}^{i,3} $
\end{alg}
\smallskip

The above process for finding tensor decomposition can in fact also be 
analogously applied for tensors $\mathcal{F}$ with rank $r \le n_2$. 
Suppose $\mathcal{F}$ has a decomposition given in \eqref{eq:findTDbygceig}.
In this case, the reduced tensor will be
 $\mathcal{T} =  P\times_1 \mathcal{F}_{1:r,1:r,:}$, where 
\begin{equation} \label{eq:pre-processing:r<n_2}
    P= \big(\mathcal{F}_{1:r,1:r,1} \big)^{-1}= \big(((U^{(2)}_{1:r,:})^\top \big)^{-1} \big((U^{(1)}_{1:r,:} \big)^{-1}.
\end{equation}
Thus,  $\mathcal{T}$ will have decomposition 
$ \mathcal{T} = \hat{U}^{(1)} \circ  U^{(2)}_{1:r,:} \circ U^{(3)}$, where
$\hat{U}^{(1)} = PU^{(1)}_{1:r,:}=((U^{(2)}_{1:r,:})^{-1})^\top $.
Then, for $k = 2, \ldots, n_3$, we would have 
\begin{equation*}
T_k= \big(((U^{(2)}_{1:r,:})^{-1} \big)^\top \mathrm{diag} \big(U^{(3)}_{k,:} \big) \big((U^{(2)}_{1:r,:} \big)^\top.
\end{equation*} 
In this case, the matrix $S$, as the common left eigenmatrix of $\{T_k\}_{k=2}^{n_3}$,
  is  $(U^{(2)}_{1:r,:})^\top$ and its rows reduce to the standard left common eigenvectors.
So, $S$ can be found by determining the left eigenvectors of $T_k$, for instance, using the power method.
Finally, the tensor decomposition of $F$ can be obtained by solving particular linear systems
analogous to Step 4 of Algorithm~\ref{alg:findr1}.
This approach is similar to the generalized eigenvalue decomposition (GEVD) method \cite{Leurgans1993,Sanchez1990}.
 
%
%
%
%
%

\subsection{The second-stage optimization algorithm}
\label{sc:secondOptimization}
In this section, we consider the scenarios where, 
instead of the entire generalized left common eigenmatrix $S$, only 
partial rows of $S$ are obtained by the first-stage Algorithm~\eqref{eq:findTDbygceig}.
Recall that $\mathcal{T}$ and $T_i \in \mathbb{C}^{r \times n_2}$ represent the tensor and its matrix slices
produced by preprocessing the original tensor $\mathcal{F}$ as in \eqref{pre-processing}. 
In this case, to find the tensor decomposition of $\mathcal{F}$, we reformulate the problem 
of solving $\{P_k\}_{k=2}^{n_3}$ by using the linear and quadratic equations \eqref{eq:linearT} and \eqref{eq:quadT}, respectively.
In what follows, for $k = 2, \ldots, n_3$, we denote
\[
T_k = \begin{pmatrix}  (T_k^1)^\top \\ (T_k^2)^\top \end{pmatrix}, \quad
 \mbox{where } T_k^1 \in \mathbb{C}^{n_2 \times n_2 } \mbox{ and } T_k^2 \in \mathbb{C}^{n_2 \times (r-n_2) }.
\]
Therefore, from \eqref{eq:linearT} we get
\[
T_i(T_j^1)^\top + P_i(T_j^2)^\top =
T_j(T_i^1)^\top  + P_j(T_i^2)^\top,
\]
which then implies
\begin{align} \label{eq:lin expand}
P_i(T_j^2)^\top - P_j(T_i^2)^\top = T_j(T_i^1)^\top-T_i(T_j^1)^\top.
\end{align}
There are a total of $ {n_3-1 \choose 2}$ choices for the pair $(i,j)$
in \eqref{eq:lin expand}. Let
\[
d_1=rn_2(n_3-1)(n_3-2)/2 \quad \mbox{ and } \quad d_2=r(r-n_2)(n_3-1) .
\]
We can reformulate \eqref{eq:lin expand} as a linear system 
\begin{equation}\label{eq:2ndOplinear1}
    A \begin{pmatrix}
    \mathrm{vec}(P_2)^\top &
    \cdots &
    \mathrm{vec}(P_{n_3})^\top \\
\end{pmatrix}^\top=b,
\end{equation}
where the coefficient matrix
\begin{equation}  \label{commute_linear_system}
    A = \begin{pmatrix}
        T^2_3 \boxtimes I_{r} & -T^2_2 \boxtimes I_{r}& 0 & \cdots & 0\\
        T^2_4\boxtimes I_{r} & 0 & -T^2_2\boxtimes I_{r} & \cdots & 0\\
        \vdots &\vdots &\vdots &  \ddots & \vdots \\
       { T^2_{n_3}}\boxtimes I_{r} & 0 & 0 & \cdots & -T^2_2\boxtimes I_{r}\\
        0 & T^2_4 \boxtimes I_{r}& -T^2_3 \boxtimes I_{r}& \cdots & 0\\
        \vdots &\vdots &\vdots &  \ddots & \vdots \\
       0 & \cdots & 0 &  T^2_{n_3} \boxtimes I_{r} &  -T^2_{n_3-1}  \boxtimes I_{r}\\
    \end{pmatrix}  \in \mathbb{C}^{ d_1\times d_2 },
\end{equation}
and the right hand side 
\begin{equation}
    b = \begin{pmatrix}
        \mathrm{vec}(T_2^1(T_3)^\top-T_3^1(T_2)^\top) \\
        \vdots \\
        \mathrm{vec}(T_{n_3-1}^1(T_{n_3})^\top-T_{n_3}^1(T_{n_3-1})^\top)\\
    \end{pmatrix}.
\end{equation}

Given any $p$ rows of the generalized left common eigenmatrix $S$ of
$\{T_k\}_{k=2}^{n_3}$, which are  the slices of the 
reduced tensor $\mathcal{T}$, the following theorem provides an important property for
 designing our second-stage optimization.
\begin{theorem} \label{thm:opt1opt2}
Let $\mathcal{F}\in \mathbb{C}^{n_1 \times n_2 \times  n_3}$ be an order-3 tensor with rank $n_2 < r \le n_1$.
Suppose $s^1,\cdots,s^p$ are linearly independent rows of the generalized left common eigenmatrix $S$ 
of the reduced tensor slices $\{T_k\}_{k=2}^{n_3} $ of $\mathcal{F}$.
Let $S^p=(s^1,\cdots,s^p)^\top$ and $P_k = M^{3,k}_{:,n_2+1:r}$, where $M^{3,k}$ is the generating matrix defined in  \eqref{def:subgenerating}.
We have
    \begin{align}\label{eq:lineareig}
    S^p\begin{pmatrix}
        T_k & P_k \\
    \end{pmatrix} = D_k S^p  ~~~\mathrm{for}~2\leq k \leq n_3,
\end{align}
where $D_k$ is a diagonal matrix. 
\end{theorem}
\begin{proof}
    Without loss of generality, let us assume $S^p = S_{1:p,:}$, where
    $S$ is the generalized left common eigenmatrix of the slices $\{T_k\}_{k=2}^{n_3}$ of 
the reduced tensor $\mathcal{T}$.
Therefore, the proof of part (i) of Theorem~\ref{thm:SequDecom} implies that
$\mathcal{T}$ has the tensor decomposition $\mathcal{T} = \hat{U}^{(1)} \circ U^{(2)} \circ U^{(3)}$,
where $\hat{U}^{(1)}$, $U^{(2)}$ and $U^{(3)}$ are given in \eqref{eq:decombyS}.
Then, by \eqref{S-Tk-1}, we have 
    \begin{equation} \label{proof:eq:STfirst}
        S^ pT_k=D_k S^p_{:,1:n_2}~\mathrm{where}~D_k = U^{(3)}_{k,1:p}.
    \end{equation}
    We can rewrite \eqref{Mjk[G]:eigeqn} of Theorem~\ref{thm:decom_to_Mjk} as
    \begin{equation} \label{M3k}
        M^{3,k}=\hat{U}^{(1)} \mathrm{diag}(U^{(3)}_{k,:})(\hat{U}^{(1)})^{-1}.
    \end{equation} 
   From \eqref{eq:decombyS}, we have $S=(\hat{U}^{(1)})^{-1}$.
   Hence, it follows from   $P_k = M^{3,k}_{:,n_2+1:r}$, \eqref{proof:eq:STfirst} and \eqref{M3k}  that 
 \begin{align}\label{proof:eq:STsecond}
    \begin{aligned}
        S^p P_k~= & ~ S_{1:p,:}  M^{3,k}_{:,n_2+1:r} \\
         ~=&~((\hat{U}^{(1)})^{-1})_{1:p,:} \hat{U}^{(1)} \mathrm{diag}(U^{(3)}_{k,:})(\hat{U}^{(1)})^{-1} (I_{r})_{:,n_2+1:r}\\
        ~=&~\mathrm{diag}(U^{(3)}_{k,1:p})((\hat{U}^{(1)})^{-1})_{1:p,n_2+1:r} \\
        ~=&~D_k S_{1:p,n_2+1:r} = D_k S^p_{:,n_2+1:r}.
    \end{aligned}
    \end{align}
 Finally,  combining \eqref{proof:eq:STfirst} and \eqref{proof:eq:STsecond}, we have \eqref{eq:lineareig} holds.
\end{proof}

From \eqref{eq:lineareig} of Theorem~\ref{thm:opt1opt2}, we have
 $S^p P_k = D_k (S^p)_{:,n_2+1:r}$, which provides additional system
  of linear equations for unknowns $\{P_k\}_{k=2}^{n_3}$.
These linear systems can be compactly written as 
\begin{equation*}
    \tilde{A} \begin{pmatrix}
    \mathrm{vec}(P_2)^\top &
    \cdots &
    \mathrm{vec}(P_{n_3})^\top \\
\end{pmatrix}^\top=\tilde{b}
\end{equation*}
where the coefficient matrix 
\begin{equation}  \label{commute_linear_system2}
    \tilde{A} = \begin{pmatrix}
        I_{r-n_2} \boxtimes S^p &0 & \cdots & 0\\
        0 & I_{r-n_2} \boxtimes S^p  & \cdots & 0\\
        \vdots & \vdots &  \ddots  & \vdots \\
         0 & \cdots & 0 & I_{r-n_2} \boxtimes S^p \\
    \end{pmatrix}  \in \mathbb{C}^{ (r-n_2)(n_3-1)p\times d_2 },
\end{equation}
and the right hand side
\begin{equation}
    \tilde{b} = \begin{pmatrix}
        \mathrm{vec}(D_2 (S^p)_{:,n_2+1:r}) \\
        \vdots \\
        \mathrm{vec}(D_{n_3} (S^p)_{:,n_2+1:r})\\
    \end{pmatrix}.
\end{equation}
Combining it with the previous linear system \eqref{eq:2ndOplinear1}, we can form a 
larger linear system for $\{P_k\}_{k=2}^{n_3}$ as
\begin{equation} \label{eq:finallinear}
     \hat{A} \begin{pmatrix}
    \mathrm{vec}(P_2)^\top &
    \cdots &
    \mathrm{vec}(P_{n_3})^\top\\
\end{pmatrix}^\top=\hat{b},
\end{equation}
where $\hat{A}=\begin{pmatrix}
        A^\top & \tilde{A}^\top 
    \end{pmatrix}^\top$ and $\hat{b}=\begin{pmatrix}
    b^\top & \tilde{b}^\top 
\end{pmatrix}^\top$.
In addition to linear system \eqref{eq:finallinear},
$\{P_k\}_{k=2}^{n_3}$ also satisfy the quadratic equations \eqref{eq:quadT}.
So, to find $\{P_k\}_{k=2}^{n_3}$,
 by denoting
\begin{equation*}
    g^{i,j}(P_i, P_{j}) \coloneqq \text{vec} \big(\begin{pmatrix}
    T_i & P_i\\
\end{pmatrix}P_j-
\begin{pmatrix}
    T_j & P_j\\
\end{pmatrix}P_i \big)
\end{equation*}  
 for $2 \leq i < j \leq n_3$,  we propose to solve the following optimization problem:
\begin{eqnarray}\label{eq:2ndOp}
        \min_{P_2,\cdots,P_{n_3} \in \mathbb{C}^{r \times (r-n_2)}} ~~~& &
        \frac{1}{2} \ \sum_{2 \leq i < j \leq n_3} \|  g^{i,j}(P_i, P_{j}) \|^2 \nonumber \\
\mbox{subject to} \qquad\quad\quad & & \hat{A} \begin{pmatrix}
    \mathrm{vec}(P_2)^\top 
    \cdots &
    \mathrm{vec}(P_{n_3})^\top \\
\end{pmatrix}^\top=\hat{b}.
\end{eqnarray}

Let $N\in \mathbb{C}^{d_2 \times d}$ be a matrix whose columns form a basis for the null space of $\hat{A}$ in \eqref{eq:finallinear}
and $\begin{pmatrix}
    \mathrm{vec}(P_2^0)^\top &
    \cdots &
    \mathrm{vec}(P_{n_3}^0)^\top \\
\end{pmatrix}^\top $ be a particular solution of \eqref{eq:finallinear}. 
Then, for $2\leq k \leq n_3$, we can parametrize the unknowns $P_k$ in  \eqref{eq:finallinear}
by $x\in \mathbb{C}^{d}$ as
\begin{equation} \label{eq:Pkx}
    \mathrm{vec}(P_k(x))=\mathrm{vec}(P^0_k)+N_kx,
\end{equation} 
where $N_k \coloneqq N_{(k-2)r(r-n_2)+1:(k-1)r(r-n_2),:}$.
So, by denoting
\begin{equation} \label{def:gx}
    g(x) \coloneqq \begin{pmatrix}
        (g^{2,3}(P_2(x),P_3(x)))^\top & \cdots & (g^{n_3-1,n_3}(P_{n_3-1}(x),P_{n_3}(x)))^\top
    \end{pmatrix},
\end{equation}
 the constrained nonlinear optimization problem \eqref{eq:2ndOp} is equivalent to
 the following unconstrained nonlinear least squares optimization:
\begin{equation}\label{eq:2ndOpx}
        \min_{x \in \mathbb{C}^{d}} ~~~\frac{1}{2}  \ \sum_{2 \leq i < j \leq n_3} 
        \|  g^{i,j}(P_i(x), P_{j}(x)) \|^2 = \frac{1}{2} \| g(x) \|^2_2.
\end{equation} 
By Theorem~\ref{thm:decom_to_Mjk}, for a generic order-3 tensor $\mathcal{F}$,
the generating matrices $\{M^{3,k}\}_{k=2}^{n_3}$ satisfy commuting equations
\eqref{eq:commute}, which is equivalent to equations
 \eqref{eq:linearT} and \eqref{eq:quadT}. 
Hence, by Theorem~\ref{thm:opt1opt2} and our construction, in the generic case,
the constrained optimization problem \eqref{eq:2ndOp} has a nonempty feasible set and
the unconstrained nonlinear least squares optimization \eqref{eq:2ndOpx}
has a global minimizer $x^*$ such that $g(x^*)=0$.

To facilitate solving the nonlinear least squares optimization \eqref{eq:2ndOpx},
 we can derive the Jacobian of the function $g(x)$ defined in \eqref{def:gx}. 
 Let $P_k=\begin{pmatrix}
    P_k^1 \\
    P_k^2 
\end{pmatrix}\in \mathbb{C}^{r \times (r-n_2)}$, where $P_k^1 \in \mathbb{C}^{n_2 \times (r-n_2)}$.
The Jacobian of the function $g$ can be given as 
\begin{equation} \label{def:Jg}
    J_g \coloneqq \begin{pmatrix}
        \frac{\partial g^{2,3}}{\partial \mathrm{vec}(P_2)}& \frac{\partial g^{2,3}}{\partial \mathrm{vec}(P_3)} & 0 & \cdots & 0\\
        \frac{\partial g^{2,4}}{\partial \mathrm{vec}(P_2)} & 0 & \frac{\partial g^{2,4}}{\partial \mathrm{vec}(P_4)}& \cdots & 0\\
        \vdots &\vdots &\vdots &  \ddots & \vdots \\
       \frac{\partial g^{2,n_3}}{\partial \mathrm{vec}(P_2)}  & 0 & 0 & \cdots & \frac{\partial g^{2,n_3}}{\partial \mathrm{vec}(P_{n_3})} \\
        0 & \frac{\partial g^{3,4}}{\partial \mathrm{vec}(P_3)} & \frac{\partial g^{3,4}}{\partial \mathrm{vec}(P_4)} & \cdots & 0\\
        \vdots &\vdots &\vdots &  \ddots & \vdots \\
       0 & \cdots & 0 &  \frac{\partial g^{n_3-1,n_3}}{\partial \mathrm{vec}(P_{n_3-1})} & \frac{\partial g^{n_3-1,n_3}}{\partial \mathrm{vec}(P_{n_3})}\\
    \end{pmatrix} \renewcommand\arraystretch{2}\begin{pmatrix}
        \frac{\partial \mathrm{vec}(P_2)}{\partial x} \\
        \frac{\partial \mathrm{vec}(P_3)}{\partial x} \\
        \vdots \\
        \frac{\partial \mathrm{vec}(P_{n_3})}{\partial x} \\
    \end{pmatrix},
\end{equation}
where 
\begin{align}
\begin{cases}
    \frac{\partial g^{i,j}}{\partial \mathrm{vec}(P_{i})} = -I_{r-n_2} \boxtimes \begin{pmatrix}
         T_j & P_j \\
     \end{pmatrix}+(P^2_j)^\top \boxtimes I_{r},\\
     \frac{\partial g^{i,j}}{\partial \mathrm{vec}(P_{j})} = I_{r-n_2} \boxtimes \begin{pmatrix}
         T_i & P_i \\
     \end{pmatrix}-(P^2_i)^\top \boxtimes I_{r},\\
     \frac{\partial \mathrm{vec}(P_i)}{\partial x} = N_k.
\end{cases} 
\end{align}

Given the Jacobian of $g$, we could solve the nonlinear least squares optimization
 \eqref{eq:2ndOpx} by a Levenberg-Marquardt-type method.
If \eqref{eq:2ndOpx} is solved with a global optimizer $x^*$ (i.e., $g(x)=0$),
the $P_k :=P_k(x^*)$ for $2 \leq k \leq n_3$ can be computed using the 
parametrization in \eqref{eq:Pkx},
which in turn can be used to recover the tensor decomposition of $\mathcal{F}$
through finding the generalized  left common eigenmatrix $S$. 
In particular, when $\begin{pmatrix}
         T_k & P_k
     \end{pmatrix},~2 \leq k \leq n_3$ are all diagonalizable, we can find the tensor decomposition based on the following theorem.
\begin{theorem}
    For a tensor $\mathcal{F} \in \mathbb{C}^{n_1 \times n_2 \times n_3}$ with $ r \leq n_1$. 
     Let $x^*$ be a global optimizer for \eqref{eq:2ndOpx} such that $g(x^*) =0$ and  $P_k :=P_k(x^*)$ for $2 \leq k \leq n_3$
 be computed using the parametrization \eqref{eq:Pkx}.
If $\begin{pmatrix}
         T_k & P_k
     \end{pmatrix}$ are all diagonalizable such that 
 $ S_k \begin{pmatrix}
             T_k & P_k
         \end{pmatrix}=D_k S_k $  for $2\leq k \leq n_3$, where $D_k$ is a diagonal matrix,
then, in the generic case,
 $S_2=\ldots=S_{n_3}$ is nonsingular and
 $\mathcal{F}$ has a tensor decomposition given in \eqref{eq:findTDbygceig}. 
 \end{theorem}
 \begin{proof}
Since $x^*$ is a global optimizer with $g(x^*) =0$,
by the construction of optimization problem \eqref{eq:2ndOpx},
the $P_k =P_k(x^*)$ for $2 \leq k \leq n_3$ must satisfy \eqref{eq:linearT} and \eqref{eq:quadT},
which by \eqref{eq:def:P} are equivalent to the commuting equations \eqref{eq:commute}. 
     Therefore, for $2\leq k \leq n_3$, all the matrices $M^{3,k} = \begin{pmatrix}
         T_k & P_k
     \end{pmatrix}$ commute pairwise.
 By our assumption, $ M^{3,k} =\begin{pmatrix}
         T_k & P_k
     \end{pmatrix}$ are all diagonalizable.
Hence, it follows from \cite[Theorem 1.3.12]{Horn_Johnson_2012} that $\begin{pmatrix}
         T_k & P_k
     \end{pmatrix}$ for $2 \leq k \leq n_3$ are simultaneously diagonalizable and have the same left eigenmatrix 
$S:=S_2=\ldots=S_{n_3}$. 
     As a result, for $2 \leq k \leq n_3$,
     \begin{align*}
         S \begin{pmatrix}
             T_k & P_k
         \end{pmatrix}=D_k S \implies  S T_k=D_k S_{:,1:n_2} 
     \end{align*}
So, $S$ is a generalized left common eigenmatrix of $T_2 \ldots T_{n_3}$.
In the generic case,  $S$ is a nonsingular. 
Hence,  $\mathcal{F}$ has a tensor decomposition given in \eqref{eq:findTDbygceig} by 
conclusion (i) of Theorem \ref{thm:SequDecom}.
 \end{proof}

By the theorem above, the matrix $S$ can be obtained by solving for the left eigenvectors of $\begin{pmatrix}
    T_k & P_k
\end{pmatrix}$ for any $k \in \{2, \ldots, n_3\}$.
We can simply take $k =2$ in our algorithm and numerical experiments.
Then, with the matrix $S$, we can get a tensor decomposition of $\mathcal{F}$ as in \eqref{eq:findTDbygceig}.
To summarize, we propose the following Algorithm~\ref{alg:findr1:step2} for finding the tensor decomposition.
We call Algorithm~\ref{alg:findr1:step2} the second-stage algorithm, as it utilizes partial
results from the first-stage Algorithm~\ref{alg:findr1} when it could not find 
the entire matrix $S$.

\smallskip
\begin{alg}
    \label{alg:findr1:step2}
The Second-Stage Algorithm for Tensor Decomposition
\smallskip

\noindent \textbf{Input:} Tensor $\mathcal{F}\in \mathbb{C}^{n_1 \times n_2 \times n_3}$
 with rank $ n_2 < r \le n_1$, the pre-processed tensor $\mathcal{T}$ and 
 the partial generalized left common eigenvectors $s^1,\cdots,s^p$ of $\{T_2\}_{k=2}^{n_3}$
 given by Algorithm~\ref{alg:findr1} with $p < r$.

\vskip 2pt 
\begin{itemize}
\item [Step~1]
Construct the linear system \eqref{eq:finallinear} using $\mathcal{T}$ and $s^1,\cdots,s^p$.

\smallskip
\item [Step~2]
Construct the function $g(x)$ as in \eqref{def:gx} and solve the nonlinear least squares 
optimization \eqref{eq:2ndOpx} with Jacobian \eqref{def:Jg} to find a global minimizer $x^*$.

\smallskip
\item [Step~3] 
Compute $P_k(x^*)$ for $2\leq k \leq n_3$ as in \eqref{eq:Pkx} using the minimizer $x^*$.

\smallskip
\item [Step~4] Compute $S=(s^1, \ldots, s^r)^\top$, whose rows are the left eigenvectors of $\begin{pmatrix}
    T_2 & P_2
\end{pmatrix}$. For $1 \leq i \leq r$ and $2 \leq k \leq n_3$, 
let $\lambda_{i,k}$ be the left eigenvalue of $\begin{pmatrix}
    T_k & P_k
\end{pmatrix}$ associated with $s^i$. Then, get a tensor decomposition of $\mathcal{T}$ as in \eqref{eq:decombyS}.

\smallskip
\item [Step~5] Solve the linear least squares and get tensor decomposition of $\mathcal{F}$ as in \eqref{eq:findTDbygceig}.

\end{itemize}
\vskip 2pt 
\textbf{Output:} A decomposition of $\mathcal{F}$:
$ \quad \mathcal{F}={U}^{(1)} \circ U^{(2)} \circ U^{(3)} $
\end{alg}
\smallskip

\section{Numerical Experiments}
\label{sc:example}
In this section, we demonstrate the performance of our two-stage (TS) optimization methods
 (Alg.~\ref{alg:findr1} and Alg.~\ref{alg:findr1:step2}).
We compare the TS method with the following methods:
\begin{itemize}
\smallskip
    \item \textbf{The classical Nonlinear Least Squares (NLS) method.}
    We used the command {\tt cpd\_nls} in the software {\tt Tensorlab} \cite{TensorLab} to apply
    the NLS method, which is a nonlinear optimization based method.
    The initial point for {\tt cpd\_nls}  is provided using {\tt cpd\_rnd} in the {\tt Tensorlab}. 

\smallskip
    \item \textbf{The classical Alternating Least Squares (ALS) method.}
    We used the command {\tt cpd\_als} in the software {\tt Tensorlab} \cite{TensorLab} to apply
    the ALS method, which is a nonlinear optimization based method.
    The initial point for {\tt cpd\_als}  is provided using {\tt cpd\_rnd} in the {\tt Tensorlab}. 
    
\smallskip    
    \item \textbf{The Normal Form (NF) method  \cite{telen2021normal}.} 
    We applied the NF method using {\tt cpd\_hnf} \cite{telen2021normal} in {\tt Julia}.
    The default setting of the code, namely the {\tt \_eigs+newton} option in \cite{telen2021normal} is used.
    This is a direct method that relates tensor decomposition problems to solving polynomial systems using linear algebra operations.
    For order-3 tensors, under the \cite[Conjecture 1]{{telen2021normal}}, the NF method
     can return a tensor decomposition if the rank
    $n_1 \geq r \le \phi (n_2-1) (n_3-1)$ for a fixed constant $\phi \in [0,1)$.
     But its computational complexity scales as $M^{\frac{5}{2}\lceil \frac{1}{1-\phi} \rceil +1}$ \cite[Theorem 1.1, Theorem 4.2]{telen2021normal}, where $M=n_1n_2n_3$.
    
\smallskip    
    \item \textbf{The method by Domanov and De Lathauwer(DDL) \cite{Domanov2017}.}
    We applied the DDL method using {\tt cpd3\_gevd} in the software {\tt Tensorlab+}\cite{tensorlabplus2025}.
    This method is also a direct method for tensor decomposition. 
    When $r > n_2$, DDL needs an integer parameter $l >0$ to construct an auxiliary matrix. 
    Larger $l$ will increase the size of the auxiliary matrix and hence, increase the computational cost.
    The DDL method can obtain tensor decompositions under certain dimension conditions 
    \cite[Theorem 8]{Domanov2017}.  
    The DDL method could practically solve tensor decomposition problems when $ r> n_1$, while the TS and NF 
    methods cannot.
\end{itemize}

\smallskip  
For easy implementation, we simply apply the built-in Levenberg–Marquardt method along 
with our provided Jacobian in {\tt MATLAB}'s {\tt fsolve} function
 to solve our first-stage and second-stage optimization problems, 
 i.e., the problems \eqref{fun:fmink} and \eqref{eq:2ndOpx},
 in Alg.~\ref{alg:findr1} and Alg.~\ref{alg:findr1:step2}, respectively.
Of course, other advanced optimization methods could be also applied for a better quality implementation.
Here, note that the NF method is implemented in {\tt Julia}, which is generally faster 
than implementing the same method in {\tt Matlab}.

 We conduct the experiments in {\tt MATLAB} R2023b on a Mac Mini m2pro chip with RAM 32GB. 
The relative backward tensor decomposition error is computed as
\[
\textbf{err-rel} \coloneqq \lVert \mathcal{F}- U^{(1)}\circ U^{(2)}\circ U^{(3)} \rVert_F/\lVert \mathcal{F} \rVert_F,
\]
where  $U^{(1)}$, $U^{(2)}$ and $U^{(3)}$ are the decomposition matrices produced by the algorithm,
and $\lVert \cdot \rVert_F$ denotes the Frobenius norm of a tensor.
In the following experiments, we consider a tensor decomposition with $\textbf{err-rel} \leq 1 \times10^{-6}$ as 
a `success' decomposition.
In the numerical result tables, ``Error" means the $\textbf{err-rel}$,
``Time" refers to the average CPU time of all successful runs
of the algorithm and ``S\_rate" denotes the success rate of 
the method for finding a correct tensor decomposition across the total runs of the algorithm.
We start with two examples involving specially designed tensors.
Then, we would test all the algorithms on randomly generated tensors. 
For those randomly generated order-3 tensors, by \cite[Theorem 1.1]{Chiantini_AnAlgorithm_2014} and \cite[Theorem 2.1]{Hauenstein2019},
 when $r=n_1$, the value $(n_2-1)(n_3-1)$ plays a critical role in determining the uniqueness of the tensor decomposition.
Therefore, we test tensors with
\begin{itemize}
    \item $n_1 = r = (n_2-1)(n_3-1)$. For our selected examples, the decomposition is unique.  
    \item $n_1 = r = (n_2-1)(n_3-1)+1$. For general tensors of such size, the number of decompositions is finite.
    \item $n_2n_3 \geq n_1 = r > (n_2-1)(n_3-1)+1$. For general tensors of such size, there are infinitely many decompositions.
\end{itemize}

\begin{example}
Consider the following tensor $\mathcal{F} \in \mathbb{C}^{5 \times 3 \times 3}$
as
\[
\mathcal{F} \coloneqq
\begin{pmatrix}
\begin{array}{rrr}
  -38 &  56 &  82 \\
   42 & 152 &  42 \\
   78 & 109 & -48 \\
  102 & -13 & -105 \\
   18 &  35 &   0
\end{array}& \rvline
&
\begin{array}{rrr}
  -55 &  126 &  92 \\
   17 &  352 &  38 \\
   93 &  226 & -63 \\
  144 & -163 & -123 \\
   27 &  -18 &  15
\end{array}&
\rvline &
\begin{array}{rrr}
   31 &  180 & -14 \\
  -77 &  434 &  88 \\
  -85 &  136 &  71 \\
   10 & -313 &  43 \\
   37 &  -96 &   1
\end{array}
\end{pmatrix}.
\]
This is a rank $5$ tensor with exact decomposition matrices 
\begin{equation}\label{eq:533TD1}
\begin{aligned}
    U^{(1)} =&
\begin{pmatrix}
  3 &  2 & -3 &  4 &  1 \\
  5 &  6 &  1 &  8 &  3 \\
  9 &  2 &  4 & -1 &  2 \\
 -3 & -5 &  5 &  1 & -2 \\
  3 & -2 &  0 &  1 & -2
\end{pmatrix},
\\
U^{(2)} =&
\begin{pmatrix}
  1 & -2 &  3 &  4 &  1 \\
  5 &  6 &  1 &  9 &  2 \\
  1 &  1 & -3 &  4 & -2
\end{pmatrix}
~\mbox{ and }~ 
U^{(3)} = 
\begin{pmatrix}
  2 &  1 &  6 &  1 & -2 \\
  3 &  5 &  7 &  1 &  3 \\
  1 &  9 & -5 &  1 &  3
\end{pmatrix}.
\end{aligned}
\end{equation}
\end{example}

By \cite[Theorem 2.1]{Hauenstein2019}, a generic tensor with
 $n_1=r = (n_2-1)(n_3-1)+1$ has $(n_2+n_3-2)!/((n_2-1)!(n_3-1)!)$ tensor decompositions. 
This implies the above tensor $\mathcal{F}$ generically has $6$ tensor decompositions.
By applying the TS method with rank $r=5$, we obtain the decomposition matrices as follows:
\begin{gather*}
\begin{aligned}
    U_{ts}^{(1)}=&\begin{pmatrix}
  0.0236 & -0.6261 &  3.8192 & -1.7013 & -0.2085 \\
  0.0472 & -1.8782 & -1.2731 & -2.8355 & -0.6255 \\
 -0.0059 & -1.2522 & -5.0923 & -5.1038 & -0.2085 \\
  0.0059 &  1.2522 & -6.3654 &  1.7013 &  0.5212 \\
  0.0059 &  1.2522 & -0.0000 & -1.7013 &  0.2085
        \end{pmatrix},\\
        U_{ts}^{(2)}=&\begin{pmatrix}
  677.3 &   3.2 & -14.1 &  -3.5 &  19.2 \\
 1524.0 &   6.4 &  -4.7 & -17.6 & -57.6 \\
  677.3 &  -6.4 &  14.1 &  -3.5 &  -9.6
        \end{pmatrix},\\
        U_{ts}^{(3)}= &\begin{pmatrix}
  1.0000 &  1.0000 &  1.0000 &  1.0000 &  1.0000 \\
  1.0000 & -1.5000 &  1.1667 &  1.5000 &  5.0000 \\
  1.0000 & -1.5000 & -0.8333 &  0.5000 &  9.0000
        \end{pmatrix}
        \end{aligned}
\end{gather*}
with the $\textbf{err-rel} = 3.6812 \times 10^{-8}$.
After permutation and rescaling, we can see that this tensor decomposition is essentially
the same as the tensor decomposition \eqref{eq:533TD1}. 
Since Algorithm \ref{alg:findr1} involves random choices of orthogonal matrices $Q_i$,
by running it multiple times, it also produces other tensor decompositions.
For example, a different set of tensor decomposition matrices given by Algorithm \ref{alg:findr1} is
\[
\begin{aligned}
      \hat{U}_{ts}^{(1)}= &\begin{pmatrix}
  -0.2410 &  -3.3695 &  14.5619 &  -1.5392 &  -0.2580 \\
  -1.2389 &  -6.5069 &   0.5072 &  -2.5653 &  -0.7354 \\
  -0.0949 &  -1.4200 &  -7.8001 &  -4.6175 &  -0.3158 \\
   0.8689 &   0.0278 & -22.4828 &   1.5392 &   0.6016 \\
   1.6429 &  -1.3645 &   2.6806 &  -1.5392 &   0.2002
        \end{pmatrix},
\end{aligned}
\]
\begin{equation}\label{eq:decom:Uts}
\begin{aligned}
        \hat{U}_{ts}^{(2)}= &\begin{pmatrix}
   3.1850 &  -5.9850 &  -4.5449 &  -8.6336 &  17.4603 \\
   6.3699 & -13.4662 &  -1.5150 & -13.4541 & -52.3810 \\
  -6.3699 &  -5.9850 &   4.5449 &   5.2862 &  -8.7302
        \end{pmatrix},\\
        \hat{U}_{ts}^{(3)}= &\begin{pmatrix}
  1.0000 &  1.0000 &  1.0000 &  1.0000 &  1.0000 \\
 -1.5000 &  1.0000 &  1.1667 &  1.7623 &  5.0000 \\
 -1.5000 &  1.0000 & -0.8333 & -0.3710 &  9.0000
        \end{pmatrix}.
\end{aligned}
\end{equation}
The above decomposition \eqref{eq:decom:Uts} is slightly different from \eqref{eq:533TD1}.
Comparing the third decomposition matrices $U^{(3)}$ and $\hat{U}_{ts}^{(3)}$,
only the columns $U^{(3)}_{:,1}=\begin{pmatrix}
    2 & 3 & 1
\end{pmatrix}^\top$ and $(\hat{U}^{(3)}_{ts})_{:,4}=\begin{pmatrix}
    1 & 1.7623 & -0.3710
\end{pmatrix}^\top$
 are different, while 
other columns of $\hat{U}^{(3)}_{ts}$ are just scalar multiple of the columns of $U^{(3)}$.
In fact, define the matrix
\begin{align} \label{eq:def:Ukr}
    U \coloneqq \begin{pmatrix}
    U^{(2)} \odot U^{(3)} & (\hat{U}^{(2)}_{ts})_{:,4} \odot (\hat{U}^{(3)}_{ts})_{:,4}
\end{pmatrix} \zhc{\in \mathbb{R}^{9 \times 6}}.
\end{align} 
We can observe that 
\begin{equation}\label{eq:Ukr}
    U \begin{pmatrix}
    -1.9491 &0.1639& -0.4397 &0.6767& -0.4004 &-1
\end{pmatrix}^\top=0.
\end{equation}
Moreover, we can check any $5$ columns of $ U$ are linearly independent.
Using any $5$ columns of $ U$, we may construct a tensor decomposition of $\mathcal{F}$.
Hence, it verifies that this tensor $\mathcal{F}$ has ${6 \choose 5}=6$ tensor decompositions.


\begin{table}[ht]
\caption{
Average CPU time, error, and success rate of TS, NF, DDL, NLS, ALS methods.}
\centering
 \begin{tabular}{||c|c|c|c|c|c||}
\hline
        & TS        & NF(\tt Julia)   & DDL & NLS       & ALS  \\ \hline
Error   & 5.9448E-09 & Fail & Fail & 3.4321E-10 & Fail \\ \hline
Time    & 0.1164     & Fail & Fail & 1.1992     & Fail \\ \hline
S\_rate & 1          & Fail & Fail & 0.6        & Fail \\ \hline
 \end{tabular}
\label{Example4.1}
\end{table}

For comparison purpose, we also apply the NF, DDL, NLS and ALS methods to solve this problem. 
We run each algorithm $10$ times, as some randomization procedures may be involved in 
the implementation of these algorithms.
The performance of the algorithms is summarized in Table~\ref{Example4.1},
where we see that the NF, DDL and ALS methods consistently fail to 
find a correct tensor decomposition. 
We believe the failure of the NF and DDL methods is due to the fact
 that this tensor does not have a unique tensor decomposition, while the failure of
the ALS method is because of converging to local minimums.
Although the NLS method sometimes finds a correct tensor decomposition, 
its success rate is only 60\%,
while our TS method always succeeds to find a tensor decomposition with much less CPU time.

\smallskip
\begin{example}
Consider the tensor $\mathcal{F} \in \mathbb{C}^{8 \times 5 \times 3}$
whose entries are given as
{\small
\[
\mathcal{F}_{i_1,i_2,i_3}= \
{\left(i_1-\frac{7}{2} \right)}^{\frac{4}{5}i_2+i_3-\frac{9}{5}}
\]}
for all $i_1,i_2,i_3$ in the corresponding range. 
Since the flatten matrix $\mathrm{Flatten(\mathcal{F},1)} $ has rank $8$, 
the rank of $\mathcal{F}$ is greater or equal to $8$.
On the other hand, since
{\small
\[
{\left(i_1-\frac{7}{2} \right)}^{\frac{4}{5}i_2+i_3-\frac{9}{5}}= \left(i_1-\frac{7}{2} \right)^{\frac{4}{5}(i_2-1)} \left(i_1-\frac{7}{2} \right)^{i_3-1}. 
\]}
we have the following rank $8$ decomposition:
\begin{align}\label{example:533TD}
\begin{aligned}
     \mathcal{F}= \sum_{s=1}^8 e_s \otimes &\begin{pmatrix}
        (1 & (s-\frac{7}{2})^{\frac{4}{5}} & (s-\frac{7}{2})^{\frac{8}{5}} & (s-\frac{7}{2})^{\frac{12}{5}} & (s-\frac{7}{2})^{\frac{16}{5}}
    \end{pmatrix}^\top \\
    \otimes &\begin{pmatrix}
        1& s-\frac{7}{2}& (s-\frac{7}{2})^{2}
    \end{pmatrix}^\top.
\end{aligned}
\end{align}
Hence, $\mathcal{F}$ is a rank $8$ tensor.
\end{example}

By \cite[Theorem 1.1]{Chiantini_AnAlgorithm_2014}, a generic tensor
with $n_1 = r = (n_2-1)(n_3-1)$ has a unique decomposition when $n_1n_2n_3 \leq 15000$.
This implies the above tensor $\mathcal{F}$ generically
has a unique decomposition.
By applying the TS method with rank $r=8$,
we obtain the decomposition matrices as follows:
{\tiny
\begin{gather*}
\begin{aligned}
    U^{(1)}= &
    \begingroup
    \setlength\arraycolsep{2pt}
    \begin{pmatrix}
0  & 0  & 0   & 0   & 0  & 0  & 0   & 0.0256 + 0.0856i  \\
0  & 0  & 0  & 0   & 0.2860 + 0.1935i  & 0  & 0   & 0   \\
0   & 0   & 0  & -0.0677 - 0.0004i & 0   & 0   & 0  & 0   \\
0   & 0   & 0.0780   & 0  & 0   & 0   & 0  & 0  \\
0   & 0   & 0   & 0   & 0   & 0   & 0.1308 - 0.0268i  & 0   \\
0   & 0   & 0  & 0  & 0  & 0.2450 - 0.0224i  & 0   & 0  \\
0  & -0.4617  & 0   & 0   & 0   & 0   & 0  & 0   \\
1   & 0  & 0  & 0  & 0  & 0  & 0   & 0 
        \end{pmatrix}\endgroup,\\
        U^{(2)}= &\begingroup
    \setlength\arraycolsep{2pt}\begin{pmatrix}
1  & -2.17  & 12.82  & -14.78 + 0.09i & 2.40 - 1.62i  & 4.05 +0.37i & 7.34 + 1.50i & 3.21 - 10.72i  \\
3.33  & -5.90  & 7.37  & 6.84 - 5.03i  & -1.37 + 3.77i & 8.43 + 0.77i & 10.15 + 2.08i & 7.71 + 21.98i  \\
11.10  & -16.08  & 4.23  & -1.48 + 4.64i & -1.53 - 5.32i & 17.54 + 1.60i & 14.04 + 2.88i & -39.87 - 27.59i \\
36.96  & -43.80  & 2.43  & -0.88 - 2.the 66i & 6.05 + 4.71i  & 36.51 + 3.33i & 19.42 + 3.98i & 100.89 - 2.33i  \\
123.11  & -119.31  & 1.40  & 1.31 + 0.94i  & -10.59 - 0.36i & 75.98 + 6.94i & 26.86 + 5.50i & -167.04 + 127.36i
        \end{pmatrix}\endgroup,\\
        U^{(3)}= &\begin{pmatrix}
1   & 1   & 1  & 1   & 1   & 1  & 1  & 1   \\
4.5   & 3.5   & 0.5  & -0.5  & -1.5  & 2.5  & 1.5  & -2.5  \\
20.25  & 12.25  & 0.25  & 0.25   & 2.25   & 6.25  & 2.25  & 6.25   \\
        \end{pmatrix}
  \end{aligned}
\end{gather*}}
with $ \textbf{err-rel} =5.8642 \times 10^{-11}$.
This tensor decomposition is in fact the same as the decomposition in \eqref{example:533TD},
since by applying a permutation matrix  $P_{\sigma}$ with $\sigma = (8, 5, 4, 3, 7, 6, 2, 1)$,
$\mathcal{F}=U^{(1)}P_{\sigma}D^{-1} \circ U^{(2)}P_{\sigma}D \circ U^{(3)}P_{\sigma}$, 
is exactly \eqref{example:533TD}, where  
 $D$ is the diagonal matrix as
\begin{eqnarray*}
D &= &\mathrm{diag}(
    0.0256 + 0.0856i, ~ 0.2860+0.1935i, ~ -0.0677 - 0.0004i, ~ 0.0780 , \\
  && \qquad  0.1308 - 0.0268i, ~0.2450 - 0.0224i, ~ -0.4617, ~1).
\end{eqnarray*}

\begin{table}[ht]
\caption{
Average CPU time, error, and success rate of TS, NF, DDL, NLS, and ALS methods.}
\centering
 \begin{tabular}{||c|c|c|c|c|c||}
\hline
        & TS         & NF(\tt Julia)         & DDL        & NLS  & ALS  \\ \hline
Error   & 1.7045E-09 & 4.5558E-16 & 4.0161E-05 & Fail & Fail \\ \hline
Time    & 0.86      & 0.011      & 0.02       & Fail & Fail \\ \hline
S\_rate & 1          & 1          & 1          & Fail & Fail \\ \hline
 \end{tabular}
\label{Example4.2}
\end{table}

As before, we also run $10$ times
the NF, DDL, NLS and ALS methods to find the tensor decomposition of $\mathcal{F}$. 
The performance of the algorithms is summarized in Table~\ref{Example4.2},
where we can see that the NLS and ALS methods consistently fail for this example,
since these methods always converge to local minimizers.
On the other hand, TS, NF and DDL methods always find a correct tensor decomposition.
For this example, we observe that the NF method uses the least CPU time and get the best accuracy. 


\begin{table}[!ht]
 \centering
 \caption{
Average CPU time, error, and success rate of TS, NF, and DDL methods
}
{\footnotesize
\begin{tabular}{||c|c|c|c|c||}
  \hline
 Dimension & Rank & \multicolumn{3}{|c||}{TS (\tt MATLAB)}\\
  \hline
  $(n_1,n_2,n_3)$ & r & Time & Error& S\_rate  \\  \hline
 (9,4,4) & 9    & 0.116 & 1.4432E-11 & 1 \\ \hline
 (16,5,5) & 16  & \textbf{0.1066}  & 2.3651E-11 & 1    \\ \hline
 (30,7,6) & 30  & \textbf{0.5944}    & 2.3091E-10  & 1    \\ \hline
 (36,7,7) & 36  & \textbf{45.485}  & \textbf{7.4882E-11} & 1    \\ \hline
 (42,8,7) & 42  & \textbf{73.171}  & 3.4716E-10 & 1    \\ \hline
 (56,9,8) & 56  & \textbf{328.91}  & \textbf{6.0562E-11} & 1    \\ \hline
 \end{tabular}
 \vspace{0.05in}
 
 \begin{tabular}{||c|c|c|c|c|c|c|c||}
 \hline
 Dimension & Rank & \multicolumn{3}{|c|}{DDL(\tt MATLAB)}& \multicolumn{3}{|c||}{NF(\tt Julia)} \\
  \hline
  $(n_1,n_2,n_3)$ & r & Time & Error& S\_rate & Time & Error & S\_rate  \\  \hline

  (9,4,4) & 9 &   0.2762  & 2.1613E-07   & 1     & \textbf{0.03467}  & \textbf{8.1644E-15}     & 1
  \\\hline
  (16,5,5) & 16 & 1.5878 & 2.7062E-07 & 0.98   &  0.3733  & \textbf{1.0004E-15}  & 1 
  \\\hline
  (30,7,6) & 30 &  \multicolumn{3}{|c|}{Fail} & 49.2177   & \textbf{2.2273E-15}      & 1
  \\\hline
  (36,7,7) & 36  & \multicolumn{3}{|c|}{Fail} 
  & 385.0445  & 7.7881E-10     & 0.98 \\\hline
   (42,8,7) & 42  &\multicolumn{3}{|c|}{Fail}
  &  2701.2269  & \textbf{5.2956E-11}     & 0.92 \\\hline
   (56,9,8) & 56 & \multicolumn{3}{|c|}{Fail} & \multicolumn{3}{|c||}{Fail}
\\\hline
 \end{tabular}
\vspace{0.05in}   
}
 \label{Table:case2}
\end{table}

In the following, we would test all the algorithms on randomly generated tensors. 
\begin{example}\label{example:r=(n2-1)(n3-1)}
In this example, we compare the performance of the TS, NF, DDL, NLS and ALS methods
for randomly generated tensors
 $\mathcal{F} \in \mathbb{R}^{n_1 \times n_2 \times n_3}$ 
with the rank $ r = n_1 = (n_2-1)(n_3-1)$. 
Again, by \cite[Theorem 1.1]{Chiantini_AnAlgorithm_2014}, these randomly generated 
tensors generically has a unique decomposition. 
For each case of  $r$ and $(n_1,n_2,n_3)$ in Table~\ref{Table:case2},
we generate $50$ instances of tensor $\mathcal{F}$.
\end{example}

In this experiment, the NLS and ALS methods always fail to find a tensor decomposition.
Therefore, we only report the results of the TS, DDL and NF methods
in Table~\ref{Table:case2}.
As shown in  Table~\ref{Table:case2}, our TS method
successfully solves almost all the testing problems, the NF method solves most
cases, while the DDL method is effective only for smaller-scale problems.
Specifically, for the instance $(n_1, n_2, n_3)=(30,7,6)$ with rank $r=30$, 
the {\tt MATLAB} implementation of the DDL method (with the default setting) crashes on our computer.
The DDL method also fails in other cases due to insufficient memory on our computer.
For example, for $(n_1, n_2, n_3)=(36,7,7)$ with rank $r=36$, it requires  
allocating an auxiliary $148{,}995 \times 148{,}995$ matrix, which needs roughly $165.4$ GB memory,
For the largest instance in Table~\ref{Table:case2}, where $(n_1, n_2, n_3)=(56,9,8)$ with rank $r=56$, 
the NF method also fails due to the extremely high memory requirements,
while the average CPU time for TS method to solve this problem only takes about $329$ seconds.
Nevertheless, as shown in Table~\ref{Table:case2}, the NF method performs very efficiently 
for solving small-scale problems. 
However, as the problem size increases, except for the first case, our TS method takes significantly less CPU time compared to
the NF method.

\begin{table}[]
\label{Table:infDecom}
 \centering
 \caption{
Average CPU time, error, and success rate of the TS method when $r=(n_2-1)(n_3-1)+1$
}

\begin{tabular}{||c|c|c|c|c||}
  \hline
 Dimension & Rank & \multicolumn{3}{|c||}{TS (\tt MATLAB)}\\
  \hline
  $(n_1,n_2,n_3)$ & r & Time & Error& S\_rate  \\  \hline
(10,4,4) & 10   & 0.0748   &  1.4307e-10 & 1    \\ \hline
(17,5,5)& 17 & 0.0722 & 7.3186e-10 & 1 \\ \hline
(31,7,6)& 31 & 0.3078   & 1.5148e-10 & 1    \\ \hline
(37,7,7)& 37 & 15.393  & 1.8573e-09 & 1 \\ \hline
(43,8,7)& 43 & 21.458   & 2.2402e-10 & 1    \\ \hline
 \end{tabular}
 
 \label{Table:case2:n_2-1n_3-1+1}
\end{table}

\begin{table}[ht]
 \centering
 \caption{
Average CPU time, error, and success rate of TS method when $r>(n_2-1)(n_3-1)+1$
}

\begin{tabular}{||c|c|c|c|c||}
  \hline
 Dimension & Rank & \multicolumn{3}{|c||}{TS (\tt MATLAB)}\\
  \hline
  $(n_1,n_2,n_3)$ & r & Time & Error& S\_rate  \\  \hline
(11,4,4) & 11   &  0.0586 & 6.8996e-11 & 1    \\ \hline
(18,5,5)& 18 & 0.0942  & 2.5945e-10  & 1 \\ \hline
(32,7,6)& 32 &  0.2174 & 1.4844e-10 & 1    \\ \hline
(39,7,7)& 39 & 10.256 & 4.3412e-10 & 1 \\ \hline
(45,8,7)& 45 & 14.206 &  2.5763e-10 & 1   \\ \hline
 \end{tabular}
 
 \label{Table:case2:n_2-1n_3-1+2}
\end{table}

\begin{example}\label{example:r>=(n2-1)(n3-1)+1}
In this example, we compare the performance of the TS, NF, DDL, NLS and ALS methods
for randomly generated tensors
 $\mathcal{F} \in \mathbb{R}^{n_1 \times n_2 \times n_3}$ 
 with rank $n_1=r>(n_2-1)(n_3-1)$. 
    By \cite[Theorem 2.1]{Hauenstein2019}, there will be finitely many
 tensor decompositions when $n_1=r=(n_2-1)(n_3-1)+1$ and
  infinitely many tensor decompositions when $n_2n_3>n_1=r>(n_2-1)(n_3-1)+1$.
    For each case of  $r$ and $(n_1,n_2,n_3)$ in Table~\ref{Table:case2:n_2-1n_3-1+1}
    and Table~\ref{Table:case2:n_2-1n_3-1+2}, we generate $50$ instances of tensor $\mathcal{F}$.
\end{example}   

Since all the tensors generated in  Table~\ref{Table:case2:n_2-1n_3-1+1}
    and Table~\ref{Table:case2:n_2-1n_3-1+2} generically do not have a unique solution,
both NF and DDL methods can not find a correct tensor decomposition. 
Moreover, in our numerical experiments, the NLS and ALS methods could not find the correct tensor decompositions of these tensors either.
However, we can see from Table~\ref{Table:case2:n_2-1n_3-1+1}
    and Table~\ref{Table:case2:n_2-1n_3-1+2} that our TS method again
always find a correct tensor decomposition in a reasonable time.
It is interesting to notice that it is usually much more efficient for TS method to solve the problems in Example  \ref{example:r>=(n2-1)(n3-1)+1} 
when the decompositions are not unique than those problems in Example \ref{example:r=(n2-1)(n3-1)} where the
tensors generically have a unique decomposition.

\section{Conclusion}
In this paper, we propose a novel two-stage optimization algorithm to solve the
order-3 tensor decomposition problem with a tensor rank that does not exceed the largest dimension. 
In the first stage, the algorithm preprocesses the tensor and focuses on finding the generalized left common eigenmatrix $S$ of the slices of the reduced tensor. In the ideal case, all the generalized left common eigenvectors of the slices can be found
and a tensor decomposition can be subsequently derived based on $S$ and solving linear least squares. 
If not all the generalized left common eigenvectors are found in the first stage, 
the second stage algorithm will use the partial rows 
of the matrix $S$ obtained from the first stage and the generating polynomials to 
recover the entire $S$.
Then, a tensor decomposition can then be constructed based on $S$ by solving linear least squares.
By comparing with other commonly used and state-of-the-art methods, our proposed two-stage optimization algorithm is highly efficient and robust for solving the
 order-3  Middle-Rank Case tensor decomposition problems, even when the tensor decompositions are not unique.

\bibliographystyle{plain} 
\bibliography{citation}

@article{Lathauwer2000,
author = {De Lathauwer, Lieven and De Moor, Bart and Vandewalle, Joos},
title = {A Multilinear Singular Value Decomposition},
journal = {SIAM Journal on Matrix Analysis and Applications},
volume = {21},
number = {4},
pages = {1253-1278},
year = {2000},
doi = {10.1137/S0895479896305696},

URL = {      https://doi.org/10.1137/S0895479896305696
},
eprint = {       https://doi.org/10.1137/S0895479896305696  

}
}

@article{BERNARDI2020,
title = {Waring, tangential and cactus decompositions},
journal = {Journal de Mathématiques Pures et Appliquées},
volume = {143},
pages = {1-30},
year = {2020},
issn = {0021-7824},
doi = {https://doi.org/10.1016/j.matpur.2020.07.003},
url = {https://www.sciencedirect.com/science/article/pii/S0021782420301203},
author = {Alessandra Bernardi and Daniele Taufer}
}

@article{Bernardi2013,
author = {Bernardi, A. and Brachat, J. and Comon, P. and Mourrain, B.},
title = {General tensor decomposition, moment matrices and applications},
year = {2013},
issue_date = {May, 2013},
publisher = {Academic Press, Inc.},
address = {USA},
volume = {52},
issn = {0747-7171},
url = {https://doi.org/10.1016/j.jsc.2012.05.012},
doi = {10.1016/j.jsc.2012.05.012},
journal = {J. Symb. Comput.},
month = may,
pages = {51–71},
numpages = {21}
}

@article{BRACHAT2010,
title = {Symmetric tensor decomposition},
journal = {Linear Algebra and its Applications},
volume = {433},
number = {11},
pages = {1851-1872},
year = {2010},
issn = {0024-3795},
doi = {https://doi.org/10.1016/j.laa.2010.06.046},
url = {https://www.sciencedirect.com/science/article/pii/S0024379510003423},
author = {Jerome Brachat and Pierre Comon and Bernard Mourrain and Elias Tsigaridas}
}

@article{SaibabaHOID2016,
author = {Saibaba, Arvind K.},
title = {HOID: Higher Order Interpolatory Decomposition for Tensors Based on Tucker Representation},
journal = {SIAM Journal on Matrix Analysis and Applications},
volume = {37},
number = {3},
pages = {1223-1249},
year = {2016},
doi = {10.1137/15M1048628},

URL = {  
        https://doi.org/10.1137/15M1048628
},
eprint = {   
        https://doi.org/10.1137/15M1048628
}
}

@article{Guha_Neuroimaging_2024,
title = "A Tensor Based Varying-Coefficient Model for Multi-Modal Neuroimaging Data Analysis",
author = "{Guha Niyogi}, Pratim and Lindquist, {Martin A.} and Tapabrata Maiti",
note = "Publisher Copyright: {\textcopyright} 1991-2012 IEEE.",
year = "2024",
doi = "10.1109/TSP.2024.3375768",
language = "English (US)",
volume = "72",
pages = "1607--1619",
journal = "IEEE Transactions on Signal Processing",
issn = "1053-587X",
publisher = "Institute of Electrical and Electronics Engineers Inc.",
}

@article{MironTDSP2020,
author = {Miron, Sebastian and Zniyed, Yassine and Boyer, Rémy and Lima Ferrer de Almeida, André and Favier, Gérard and Brie, David and Comon, Pierre},
title = {Tensor methods for multisensor signal processing},
journal = {IET Signal Processing},
volume = {14},
number = {10},
pages = {693-709},
doi = {https://doi.org/10.1049/iet-spr.2020.0373},
url = {https://ietresearch.onlinelibrary.wiley.com/doi/abs/10.1049/iet-spr.2020.0373},
eprint = {https://ietresearch.onlinelibrary.wiley.com/doi/pdf/10.1049/iet-spr.2020.0373},
year = {2020}
}

@ARTICLE{Sidiropoulos_TD_SP2017,
  author={Sidiropoulos, Nicholas D. and De Lathauwer, Lieven and Fu, Xiao and Huang, Kejun and Papalexakis, Evangelos E. and Faloutsos, Christos},
  journal={IEEE Transactions on Signal Processing}, 
  title={Tensor Decomposition for Signal Processing and Machine Learning}, 
  year={2017},
  volume={65},
  number={13},
  pages={3551-3582},
  doi={10.1109/TSP.2017.2690524}}

@article{Domanov2017,
author = {Ignat Domanov and Lieven De Lathauwer},
title = {Canonical polyadic decomposition of third-order tensors: Relaxed uniqueness conditions and algebraic algorithm},
journal = {Linear Algebra and its Applications},
volume = {513},
pages = {342-375},
year = {2017},
issn = {0024-3795},
doi = {https://doi.org/10.1016/j.laa.2016.10.019},
url = {https://www.sciencedirect.com/science/article/pii/S002437951630492X}
}

@article{telen2021normal,
author = {Telen, Simon and Vannieuwenhoven, Nick},
title = {A Normal Form Algorithm for Tensor Rank Decomposition},
year = {2022},
issue_date = {December 2022},
publisher = {Association for Computing Machinery},
address = {New York, NY, USA},
volume = {48},
number = {4},
issn = {0098-3500},
url = {https://doi.org/10.1145/3555369},
doi = {10.1145/3555369},
journal = {ACM Trans. Math. Softw.},
month = {December},
articleno = {38},
numpages = {35}
}

@misc{tensorlabplus2025,
  Title     = {Tensorlab\textsuperscript{+}},
  Author    = {Hendrikx, S. and Bouss\'e, M. and Vervliet, N. and Vandecappelle, M. and Kenis, 
              R. and De Lathauwer, L.},
  Note      = {Available online, Version of June 2025 downloaded from \url{https://www.tensorlabplus.net}},
}

@INPROCEEDINGS{TensorLab,
  author={Vervliet, Nico and Debals, Otto and De Lathauwer, Lieven},
  booktitle={2016 50th Asilomar Conference on Signals, Systems and Computers}, 
  title={Tensorlab 3.0 — Numerical optimization strategies for large-scale constrained and coupled matrix/tensor factorization}, 
  year={2016},
  volume={},
  number={},
  pages={1733-1738},
  doi={10.1109/ACSSC.2016.7869679}}

@article{Communitydetection_JLLD,
author = {Jing, Bing-Yi and Li, Ting and Lyu, Zhongyuan and Xia, Dong},
year = {2021},
month = {December},
pages = {},
title = {Community detection on mixture multilayer networks via regularized tensor decomposition},
volume = {49},
journal = {The Annals of Statistics},
doi = {10.1214/21-AOS2079}
}

@article{learninglatentvariable,
author = {Anandkumar, Animashree and Ge, Rong and Hsu, Daniel and Kakade, Sham M. and Telgarsky, Matus},
title = {Tensor decompositions for learning latent variable models},
year = {2014},
issue_date = {January 2014},
publisher = {JMLR.org},
volume = {15},
number = {1},
issn = {1532-4435},
journal = {J. Mach. Learn. Res.},
month = {January},
pages = {2773–2832},
numpages = {60}
}

@article{RungangExactClustering,
    author = {Han, Rungang and Luo, Yuetian and Wang, Miaoyan and Zhang, Anru R.},
    title = "{Exact Clustering in Tensor Block Model: Statistical Optimality and Computational Limit}",
    journal = {Journal of the Royal Statistical Society Series B: Statistical Methodology},
    volume = {84},
    number = {5},
    pages = {1666-1698},
    year = {2022},
    month = {October},
    issn = {1369-7412},
    doi = {10.1111/rssb.12547},
    url = {https://doi.org/10.1111/rssb.12547},
    eprint = {https://academic.oup.com/jrsssb/article-pdf/84/5/1666/49459279/jrsssb\_84\_5\_1666.pdf},
}

@article{han2022optimal,
  title={An optimal statistical and computational framework for generalized tensor estimation},
  author={Han, Rungang and Willett, Rebecca and Zhang, Anru R},
  journal={The Annals of Statistics},
  volume={50},
  number={1},
  pages={1--29},
  year={2022},
  publisher={Institute of Mathematical Statistics}
}

@article{Hauenstein2019,
url = {https://doi.org/10.1515/crelle-2016-0067},
title = {Homotopy techniques for tensor decomposition and perfect identifiability},
author = {Jonathan D. Hauenstein and Luke Oeding and Giorgio Ottaviani and Andrew J. Sommese},
pages = {1--22},
volume = {2019},
number = {753},
journal = {Journal für die reine und angewandte Mathematik (Crelles Journal)},
doi = {doi:10.1515/crelle-2016-0067},
year = {2019}
}

@article{KUO2018homotopy,
title = {Computing the unique CANDECOMP/PARAFAC decomposition of unbalanced tensors by homotopy method},
journal = {Linear Algebra and its Applications},
volume = {556},
pages = {238-264},
year = {2018},
issn = {0024-3795},
doi = {https://doi.org/10.1016/j.laa.2018.07.004},
url = {https://www.sciencedirect.com/science/article/pii/S0024379518303227},
author = {Yueh-Cheng Kuo and Tsung-Lin Lee}
}

@article{Sanchez1990,
author = {Sanchez, Eugenio and Kowalski, Bruce R.},
title = {Tensorial resolution: A direct trilinear decomposition},
journal = {Journal of Chemometrics},
volume = {4},
number = {1},
pages = {29-45},
doi = {https://doi.org/10.1002/cem.1180040105},
url = {https://analyticalsciencejournals.onlinelibrary.wiley.com/doi/abs/10.1002/cem.1180040105},
eprint = {https://analyticalsciencejournals.onlinelibrary.wiley.com/doi/pdf/10.1002/cem.1180040105},
year = {1990}
}

@article{HuaTensorRegression,
author = {Hua Zhou and Lexin Li and Hongtu Zhu},
title = {Tensor Regression with Applications in Neuroimaging Data Analysis},
journal = {Journal of the American Statistical Association},
volume = {108},
number = {502},
pages = {540--552},
year = {2013},
publisher = {ASA Website},
doi = {10.1080/01621459.2013.776499},
    note ={PMID: 24791032},
URL = {       https://doi.org/10.1080/01621459.2013.776499
},
eprint = {       https://doi.org/10.1080/01621459.2013.776499
}
}

@article{Leurgans1993,
author = {Leurgans, S. E. and Ross, R. T. and Abel, R. B.},
title = {A Decomposition for Three-Way Arrays},
journal = {SIAM Journal on Matrix Analysis and Applications},
volume = {14},
number = {4},
pages = {1064-1083},
year = {1993},
doi = {10.1137/0614071},
URL = { 
    https://doi.org/10.1137/0614071
},
eprint = {   
        https://doi.org/10.1137/0614071
}

}

@article{TensorsStatistics,
author = {Bi, Xuan and Tang, Xiwei and Yuan, Yubai and Zhang, Yanqing and Qu, Annie},
title = {Tensors in Statistics},
journal = {Annual Review of Statistics and Its Application},
volume = {8},
number = {1},
pages = {345-368},
year = {2021},
doi = {10.1146/annurev-statistics-042720-020816},

URL = { 
    
        https://doi.org/10.1146/annurev-statistics-042720-020816
    
    

},
eprint = { 
    
        https://doi.org/10.1146/annurev-statistics-042720-020816
    
    

}
}

@article{DynamicTensor,
author = {Zhang, Yanqing and Bi, Xuan and Tang, Niansheng and Qu, Annie},
title = {Dynamic Tensor Recommender Systems},
year = {2021},
issue_date = {January 2021},
publisher = {JMLR.org},
volume = {22},
number = {1},
issn = {1532-4435},
journal = {J. Mach. Learn. Res.},
month = {January},
articleno = {65},
numpages = {35}
}

@misc{pereira2022tensor,
      title={Tensor Moments of Gaussian Mixture Models: Theory and Applications}, 
      author={João M. Pereira and Joe Kileel and Tamara G. Kolda},
      year={2022},
      eprint={2202.06930},
      archivePrefix={arXiv},
      primaryClass={stat.ML},
      url={https://arxiv.org/abs/2202.06930}, 
}

@misc{liu2023tensor,
      title={Tensor Decomposition for Model Reduction in Neural Networks: A Review}, 
      author={Xingyi Liu and Keshab K. Parhi},
      year={2023},
      eprint={2304.13539},
      archivePrefix={arXiv},
      primaryClass={cs.LG}
}

@misc{girka2023tensor,
      title={Tensor Generalized Canonical Correlation Analysis}, 
      author={Fabien Girka and Arnaud Gloaguen and Laurent Le Brusquet and Violetta Zujovic and Arthur Tenenhaus},
      year={2023},
      eprint={2302.05277},
      archivePrefix={arXiv},
      primaryClass={stat.ML}
}

@misc{minster2021cp,
      title={CP Decomposition for Tensors via Alternating Least Squares with QR Decomposition}, 
      author={Rachel Minster and Irina Viviano and Xiaotian Liu and Grey Ballard},
      year={2021},
      eprint={2112.10855},
      archivePrefix={arXiv},
      primaryClass={math.NA}
}

@article{TensorToolbox,
author = {Bader, Brett W. and Kolda, Tamara G.},
title = {Algorithm 862: MATLAB Tensor Classes for Fast Algorithm Prototyping},
year = {2006},
issue_date = {December 2006},
publisher = {Association for Computing Machinery},
address = {New York, NY, USA},
volume = {32},
number = {4},
issn = {0098-3500},
url = {https://doi.org/10.1145/1186785.1186794},
doi = {10.1145/1186785.1186794},
journal = {ACM Trans. Math. Softw.},
month = {December},
pages = {635–653},
numpages = {19}
}

@article{TensorLy,
  author  = {Jean Kossaifi and Yannis Panagakis and Anima Anandkumar and Maja Pantic},
  title   = {TensorLy: Tensor Learning in Python},
  journal = {Journal of Machine Learning Research},
  year    = {2019},
  volume  = {20},
  number  = {26},
  pages   = {1--6},
  url     = {http://jmlr.org/papers/v20/18-277.html}
}

@article{DomanovLathauwer2015,
author = {Domanov, Ignat and Lathauwer, Lieven De},
title = {Generic Uniqueness Conditions for the Canonical Polyadic Decomposition and INDSCAL},
journal = {SIAM Journal on Matrix Analysis and Applications},
volume = {36},
number = {4},
pages = {1567-1589},
year = {2015},
doi = {10.1137/140970276},
URL = {    
        https://doi.org/10.1137/140970276
},
eprint = {   
        https://doi.org/10.1137/140970276
}
}

@article{DomanovLathauwer2013,
author = {Domanov, Ignat and De Lathauwer, Lieven},
title = {On the Uniqueness of the Canonical Polyadic Decomposition of  Third-Order Tensors---Part I: Basic Results and Uniqueness of One Factor Matrix},
year = {2013},
publisher = {Society for Industrial and Applied Mathematics},
address = {USA},
volume = {34},
number = {3},
issn = {0895-4798},
url = {https://doi.org/10.1137/120877234},
doi = {10.1137/120877234},
journal = {SIAM J. Matrix Anal. Appl.},
month = {January},
pages = {855–875},
numpages = {21}
}

@ARTICLE{Cichocki_2015,
  author={Cichocki, Andrzej and Mandic, Danilo and De Lathauwer, Lieven and Zhou, Guoxu and Zhao, Qibin and Caiafa, Cesar and PHAN, HUY ANH},
  journal={IEEE Signal Processing Magazine}, 
  title={Tensor Decompositions for Signal Processing Applications: From two-way to multiway component analysis}, 
  year={2015},
  volume={32},
  number={2},
  pages={145-163},
  doi={10.1109/MSP.2013.2297439}}

@article{guo2021learning,
author = {Guo, Bingni and Nie, Jiawang and Yang, Zi},
year = {2021},
month = {November},
pages = {},
title = {Learning Diagonal Gaussian Mixture Models and Incomplete Tensor Decompositions},
volume = {50},
journal = {Vietnam Journal of Mathematics},
doi = {10.1007/s10013-021-00534-3}
}

@article{RHODES20101818,
title = {A concise proof of Kruskal’s theorem on tensor decomposition},
journal = {Linear Algebra and its Applications},
volume = {432},
number = {7},
pages = {1818-1824},
year = {2010},
issn = {0024-3795},
doi = {https://doi.org/10.1016/j.laa.2009.11.033},
url = {https://www.sciencedirect.com/science/article/pii/S0024379509006132},
author = {John A. Rhodes}
}

@article{larsen2020practical,
author = {Larsen, Brett W. and Kolda, Tamara G.},
title = {Practical Leverage-Based Sampling for Low-Rank Tensor Decomposition},
journal = {SIAM Journal on Matrix Analysis and Applications},
volume = {43},
number = {3},
pages = {1488-1517},
year = {2022},
doi = {10.1137/21M1441754},
URL = {     
        https://doi.org/10.1137/21M1441754
},
eprint = {   
        https://doi.org/10.1137/21M1441754
}
}

@inproceedings{bhaskara2014smoothed,
author = {Bhaskara, Aditya and Charikar, Moses and Moitra, Ankur and Vijayaraghavan, Aravindan},
title = {Smoothed analysis of tensor decompositions},
year = {2014},
isbn = {9781450327107},
publisher = {Association for Computing Machinery},
address = {New York, NY, USA},
url = {https://doi.org/10.1145/2591796.2591881},
doi = {10.1145/2591796.2591881},
booktitle = {Proceedings of the Forty-Sixth Annual ACM Symposium on Theory of Computing},
pages = {594–603},
numpages = {10},
location = {New York, New York},
series = {STOC '14}
}

@article{HillarNPHard2013,
author = {Hillar, Christopher J. and Lim, Lek-Heng},
title = {Most Tensor Problems Are NP-Hard},
year = {2013},
issue_date = {November 2013},
publisher = {Association for Computing Machinery},
address = {New York, NY, USA},
volume = {60},
number = {6},
issn = {0004-5411},
url = {https://doi.org/10.1145/2512329},
doi = {10.1145/2512329},
journal = {J. ACM},
month = {November},
articleno = {45},
numpages = {39}
}

@article{Phanreshaping2013,
author = {Phan, Anh-Huy and Tichavsky, Petr and Cichocki, Andrzej},
year = {2013},
month = {October},
pages = {},
title = {CANDECOMP/PARAFAC decomposition of high-order tensors through tensor reshaping},
journal = {IEEE Transactions on Signal Processing},
doi = {10.1109/TSP.2013.2269046}
}

@book{Horn_Johnson_2012, place={Cambridge}, edition={2}, title={Matrix Analysis}, publisher={Cambridge University Press}, author={Horn, Roger A. and Johnson, Charles R.}, year={2012}}

@article{Chiantini_AnAlgorithm_2014,
author = {Chiantini, Luca and Ottaviani, Giorgio and Vannieuwenhoven, Nick},
title = {An Algorithm For Generic and Low-Rank Specific Identifiability of Complex Tensors},
journal = {SIAM Journal on Matrix Analysis and Applications},
volume = {35},
number = {4},
pages = {1265-1287},
year = {2014},
doi = {10.1137/140961389},
URL = { 
        https://doi.org/10.1137/140961389
},
eprint = { 
        https://doi.org/10.1137/140961389
}
}

@article{Chiantini_2017,
author = {Chiantini, Luca and Ottaviani, Giorgio and Vannieuwenhoven, Nick},
title = {Effective Criteria for Specific Identifiability of Tensors and Forms},
journal = {SIAM Journal on Matrix Analysis and Applications},
volume = {38},
number = {2},
pages = {656-681},
year = {2017},
doi = {10.1137/16M1090132},

URL = {     
        https://doi.org/10.1137/16M1090132       
},
eprint = { 
    
        https://doi.org/10.1137/16M1090132
}
}

@article{nie2015generating,
author = {Nie, Jiawang},
title = {Generating Polynomials and Symmetric Tensor Decompositions},
year = {2017},
issue_date = {April 2017},
publisher = {Springer-Verlag},
address = {Berlin, Heidelberg},
volume = {17},
number = {2},
issn = {1615-3375},
url = {https://doi.org/10.1007/s10208-015-9291-7},
doi = {10.1007/s10208-015-9291-7},
journal = {Found. Comput. Math.},
month = {April},
pages = {423–465},
numpages = {43}
}

@article{nie2020hermitian,
author = {Nie, Jiawang and Yang, Zi},
title = {Hermitian Tensor Decompositions},
journal = {SIAM Journal on Matrix Analysis and Applications},
volume = {41},
number = {3},
pages = {1115-1144},
year = {2020},
doi = {10.1137/19M1306889},

URL = {  
        https://doi.org/10.1137/19M1306889  
},
eprint = {     
        https://doi.org/10.1137/19M1306889  
}
}

@article{lovitz2021generalization, title={A generalization of Kruskal’s theorem on tensor decomposition}, volume={11}, DOI={10.1017/fms.2023.20}, journal={Forum of Mathematics, Sigma}, author={Lovitz, Benjamin and Petrov, Fedor}, year={2023}, pages={e27}}

@article{NWZ22,
author = {Nie, Jiawang and Wang, Li and Zheng, Zequn},
year = {2023},
month = {January},
volume={19},
pages = {237-255},
title = {Higher Order Correlation Analysis for Multi-View Learning},
journal={Pacific Journal of Optimization}, 
}

@misc{NieLR14,
      title={Nearly Low Rank Tensors and Their Approximations}, 
      author={Jiawang Nie},
      year={2014},
      eprint={1412.7270},
      archivePrefix={arXiv},
      primaryClass={math.NA},
      url={https://arxiv.org/abs/1412.7270}, 
}

@article{NWZ23,
author = {Nie, Jiawang and Wang, Li and Zheng, Zequn},
year = {2023},
month = {March},
pages = {},
title = {Low Rank Tensor Decompositions and Approximations},
journal = {Journal of the Operations Research Society of China},
doi = {10.1007/s40305-023-00455-7}
}
\end{document}